\documentclass[]{amsart}

\usepackage{amsmath,amsthm}
\usepackage{amsfonts}
\usepackage[linesnumbered,lined,boxed,commentsnumbered]{algorithm2e}
\usepackage{booktabs}
\usepackage[
  mode=match,
  propagate-math-font=true,
  reset-math-version=false,
  reset-text-family=false,
  reset-text-series=false,
  text-family-to-math=true,
  text-series-to-math=true]{siunitx}
\usepackage{caption}
\usepackage{tikz}
\usepackage{subcaption}
\usetikzlibrary{matrix, calc, arrows}
\usepackage{upgreek}
\usepackage{bm}
\usepackage{url}
\usepackage{longtable}
\usepackage{afterpage}
\usepackage{xspace}  % add space after the command
\usepackage{hyperref}

\definecolor{titlecolor}{RGB}{245, 242, 238} 
\definecolor{drawgray}{HTML}{666666}
\definecolor{drawgray_l}{HTML}{F5F5F5}
\definecolor{drawblue}{HTML}{6C8EBF}
\definecolor{drawblue_l}{HTML}{DAE8FC}
\definecolor{drawgreen}{HTML}{82B366}
\definecolor{drawgreen_l}{HTML}{D5E8D4}
\definecolor{draworange}{HTML}{D79B00}
\definecolor{draworange_l}{HTML}{FFE6CC}
\definecolor{drawyellow}{HTML}{D6B656}
\definecolor{drawyellow_l}{HTML}{FFF2CC}
\definecolor{drawred}{HTML}{B85450}
\definecolor{drawred_l}{HTML}{EA6B66}
\definecolor{drawviolet}{HTML}{9673A6}
\definecolor{drawviolet_l}{HTML}{E1D5E7}
\hypersetup{
    colorlinks=true,
    linkcolor=drawred_l,
    breaklinks=true,
    citecolor=teal,
    urlcolor=drawblue,
    menucolor=green,
    runcolor=blue,
}

\usetikzlibrary{shapes,arrows}
\tikzstyle{block} = [rectangle, draw, fill=blue!20, text width=4em, text centered, rounded corners, minimum height=4em]
\tikzstyle{decision} = [diamond, draw, fill=blue!20, text width=4.5em, text badly centered, node distance=3cm, inner sep=0pt]
\tikzstyle{line} = [draw, -latex']
\renewcommand{\epsilon}{\varepsilon}
\newcommand{\GAP}{\text{GAP}}
\newcommand{\abs}[1]{\left\lvert #1 \right\rvert}
\newcommand{\size}{k}
\newcommand{\maxiter}{n_\textup{max}}
\newcommand{\tr}{$\omega_\textsc{tr}$\xspace}
\newcommand{\eu}{$\omega_\textsc{eu}$\xspace}
\newcommand{\ew}{$\omega_\textsc{ew}$\xspace}
\newcommand{\rbp}{$r_\textsc{bp}$\xspace}
\newcommand{\rcg}{$r_\textsc{cg}$\xspace}
\newcommand{\scalingFactor}{0.95}
\newcommand{\param}{\mathit}
\newcommand{\argmin}{\text{argmin}}
\newcommand{\bpplain}{\textsc{bp}\xspace}
\newcommand{\bpmulti}{\textsc{bp\_m}\xspace}
\newcommand{\bpbg}{\textsc{bp+b}\xspace}
\newcommand{\lns}{\textsc{lns}\xspace}
\newcommand{\lnsr}{\textsc{lns+r(b)}\xspace}
\newcommand{\lnsrb}{\textsc{lns+rb(b)}\xspace}
\newcommand{\lnsrf}{\textsc{lns+r(f)}\xspace}
\newcommand{\lnsrbf}{\textsc{lns+rb(f)}\xspace}

\newcommand{\lnsbf}{\textsc{lns+b(f)}\xspace}

\theoremstyle{plain}
\newtheorem{lemma}{Lemma}

\usepackage{glossaries} 
\glsdisablehyper % disable the acronyms hyperlinking
\newacronym{bdsp}{BDSP}{Bus Driver Scheduling Problem}
\newacronym{bp}{B\&P}{Branch and Price}
\newacronym{cg}{CG}{Column Generation}
\newacronym{lns}{LNS}{Large Neighborhood Search}
\newacronym{rcspp}{RCSPP}{Resource Constrained Shortest Path Problem}

\begin{document}

% % Title and author information
\title[Integrating CG and LNS for the BDSP]{Integrating Column Generation and Large Neighborhood Search for Bus Driver Scheduling with Complex Break Constraints}

\author[Kletzander]{Lucas Kletzander}
\email{lucas.kletzander@tuwien.ac.at}
\author[Mannelli Mazzoli]{Tommaso Mannelli Mazzoli}
\email{tommaso.mazzoli@tuwien.ac.at}
\author[Musliu]{Nysret Musliu}
\email{nysret.musliu@tuwien.ac.at}
\author[Van Hentenryck]{Pascal Van Hentenryck}
\email{pascal.vanhentenryck@isye.gatech.edu}

% % Abstract
\begin{abstract}
    The \gls{bdsp} is a combinatorial optimization problem with the goal to design shifts to cover prearranged bus tours. The objective takes into account the operational cost as well as the satisfaction of drivers. This problem is heavily constrained due to strict legal rules and collective agreements.
    The objective of this article is to provide state-of-the-art exact and hybrid solution methods that can provide high-quality solutions for instances of different sizes.
    This work presents a comprehensive study of both an exact method, \gls{bp}, as well as a \gls{lns} framework which uses \gls{bp} or \gls{cg} for the repair phase to solve the \gls{bdsp}. It further proposes and evaluates a novel deeper integration of \gls{bp} and \gls{lns}, storing the generated columns from the \gls{lns} subproblems and reusing them for other subproblems, or to find better global solutions. 
    The article presents a detailed analysis of several components of the solution methods and their impact, including general improvements for the \gls{bp} subproblem, which is a high-dimensional \gls{rcspp}, and the components of the \gls{lns}. The evaluation shows that our approach provides new state-of-the-art results for instances of all sizes, including exact solutions for small instances, and low gaps to a known lower bound for mid-sized instances.
    We observe that \gls{bp} provides the best results for small instances, while the tight integration of \gls{lns} and \gls{cg} can provide high-quality solutions for larger instances, further improving over \gls{lns} which just uses \gls{cg} as a black box. The proposed methods are general and can also be applied to other rule sets and related optimization problems.
\end{abstract}

\maketitle

\section{Introduction}
\label{sec:Introduction}

Scheduling employees for public transport is an area that needs to respect a variety of complex constraints, time schedules, spatial requirements, and conflicting objectives, giving rise to difficult optimization problems. Drivers bear great responsibility for their passengers, leading to a range of legal requirements, collective agreements and company policies that demand particular break assignments and shift characteristics. Companies require cost-efficient schedules, while drivers and labor unions request employee-friendly schedules to reduce stress and increase compatibility of the required shift work with the private life of the drivers.

This paper is a significantly extended version of our previous conference papers presented at AAAI \cite{kletzander2021branch} and ICAPS \cite{mazzoli2024investigating}. It presents exact and hybrid methods for a complex \acrfull{bdsp}. The constraints of the \gls{bdsp} depend on the country's legal regulation, and in our case, they follow the Austrian \textit{collective agreement for employees in private omnibus providers}~\cite{WKO} using the rules for regional lines. In particular, the collective agreement has stringent rules requiring drivers to take frequent breaks, with the eventual option of splitting them into multiple parts, and the objectives in this article combine cost optimization with several employee satisfaction criteria, such as reducing long unpaid breaks or frequent vehicle changes.

As an exact approach, we propose \acrfull{bp}. This method works on a branching tree, and performs \acrfull{cg} at each node, where the problem is split into Set Partitioning as the master problem and the \acrfull{rcspp} as the subproblem. However, due to the complex set of constraints, the labels used in solving the subproblem have a high number of dimensions, leading to a slow determination of non-dominated labels when solving the subproblem. Therefore, several options to improve the generation of new columns in the presence of a high number of dimensions are investigated, ultimately leading to the use of k-d trees and bounding boxes in a two-stage dominance algorithm. This approach was originally published in a conference paper at AAAI \cite{kletzander2021branch}.

Although \gls{bp} gives very good results for small- to medium-sized instances, exact methods are computationally too expensive for large instances, and heuristic methods cannot obtain optimal solutions. Therefore, the study of new methods to tackle larger instances is of particular interest. In this paper, we present a novel approach based on the \acrfull{lns} framework, which has been successfully used for solving other challenging real-life problems. However, applying this method to the \gls{bdsp} requires innovative ideas for the destroy and repair operators, as well as a detailed investigation into their impact and the choice of parameters. In particular, \gls{bp} (or the \gls{cg} part of \gls{bp}) is used as the repair operator, as it provides high-quality solutions for small sub-problems very quickly. This approach was originally published in a conference paper at ICAPS \cite{mazzoli2024investigating}.

However, it is possible to see \gls{bp} and \gls{lns} not only as separate components, but to enforce a tighter integration between the two. Therefore, this article is a substantial extension of the previous conference papers, investigating how to boost performance even further by allowing exchange between the different sub-problems of \gls{lns}, in particular, by storing columns from sub-problems for use in later, different sub-problems, and by using the combined set of columns that are aggregated from different sub-problems to find better solutions even more quickly. This leads to significant improvements in the results, establishing this hybrid as the new state-of-the-art method for this version of the \gls{bdsp}.
 
The main contributions of this paper are:
\begin{itemize}
    \item A \acrfull{bp} approach provides an exact solution method that can solve small instances optimally, and provide high-quality solutions for medium size.
    \item Several improvements in solving the \acrfull{rcspp} in general are proposed and evaluated, which allow solving high-dimensional \glspl{rcspp} much faster.
	\item A novel \acrfull{lns} is proposed method, including innovative destroy operators, using the \gls{cg} part of our \gls{bp} as a repair.
    \item A detailed analysis of the effect of the different destroy and repair components of \gls{lns} is provided.
    \item We propose a novel tight integration of \gls{lns} and \gls{cg}, where columns from each sub-problem are aggregated and reused to significantly improve the search.
    \item A detailed evaluation is provided on the complex Austrian \gls{bdsp} using a public benchmark instance set, where our approach clearly constitutes the new state of the art, outperforming previous results across all sizes of instances. While this evaluation is specific to this version of the problem, the ideas used in \gls{bp} and \gls{lns} and their integration are generally applicable for complex personnel scheduling problems. 
\end{itemize}

The remainder of this article is structured as follows. Section~\ref{sec:rw} provides related work to the problem and methods, Section~\ref{sec:ProblDef} introduces the formal definition of the problem, Section~\ref{sec:bp} contains the details about \acrlong{bp}, Section~\ref{sec:Algo} explains the details about \acrlong{lns}, Section~\ref{sec:integration} provides the tight integration of the two approaches, Section~\ref{sec:results} contains the experiments for the different approaches and their detailed analysis, and finally Section~\ref{sec:conclusion} provides conclusions and an outlook to future work.

\section{Related Work}
\label{sec:rw}

There are many versions of employee scheduling problems and several surveys \cite{ernststaff2004,vandenberghpersonnel2013} provide a good overview of this line of work. Driver scheduling, as a part of crew scheduling, resides between vehicle scheduling and driver rostering in the process of operating bus transport systems \cite{ibarra-rojasplanning2015}.

Research on bus driver scheduling problems has a long history \cite{fandelbus1995} and uses a variety of solution methods. Exact methods mostly use column generation with a set covering or set partitioning master problem and a resource constrained shortest path subproblem  \cite{smith1988bus,desrocherscolumn1989,portugaldriver2009,lincolumn2016}. 

Heuristic methods like greedy \cite{martelloheuristic1986,deleonebus2011,tothefficient2013} or exhaustive \cite{chencrew2013} search, tabu search \cite{lourencomultiobjective2001,fandeltabu2001}, genetic algorithms \cite{lourencomultiobjective2001,lifuzzy2003}, or iterated assignment problems \cite{constantino2017solving} are used in different variations.

However, most work so far focuses mainly on cost, rarely minimizing idle time and vehicle changes \cite{ibarra-rojasplanning2015,constantino2017solving}. Break constraints are mostly simple, often including just one meal break. Complex break scheduling within shifts has been considered by authors in different contexts \cite{beerscheduling2008,beerai-based2010,widlbreak2014}. There is not much work on multi-objective bus driver scheduling \cite{lourencomultiobjective2001}, but multi-objective approaches are used in other bus operation problems \cite{respicio2013enhanced}. Recent work includes automated weight setting \cite{kletzander2023dynamic} and decision support for human planners \cite{frohner2024decision}.

This article investigates a complex real-life Bus Driver Scheduling Problem that was first introduced by \cite{kletzander2020solving}. This work explains the need for a combined objective function that goes beyond cost. The published instances were solved with simulated annealing and a hill climbing heuristic. Good results have also been provided using tabu search \cite{Kletzander2022} and hyper-heuristics \cite{Kletzander2022,Kletzander2023last}, while the state-of-the-art heuristic solutions beside the methods in this article are achived using Construct, Solve, Merge, and Adapt \cite{Rosati2023}.

Branch and Price is a decomposition technique for large mixed integer programs \cite{barnhart1998branch}, where a master problem works on a set of columns, while a sub-problem is responsible for generating new columns with negative reduced cost, i.e., the potential to improve the solution to the master problem. This work uses set partitioning \cite{balas1976set} as the master problem and the RCSPP \cite{irnich2005shortest} as the subproblem. Resources are modeled via resource extension functions (REF) \cite{irnich2008resource}. Different techniques for dealing with pareto-front calculation can mostly be found as maxima-finding algorithms in a geometrical context \cite{bentley1980multidimensional,chen2012maxima}. Large neighborhood search \cite{Shaw1998} is an iterative solution method, where in each iteration a part of the solution is destroyed and rebuilt by dedicated destroy and repair operators.

\section{Problem Description} \label{sec:ProblDef}

The investigated \acrfull{bdsp} deals with the assignment of bus drivers to vehicles that already have a predetermined route for one day of operation. The problem specification was introduced by \cite{kletzander2020solving}. 
\subsection{Problem Input}
The input of the \gls{bdsp} consists of three pieces of data:
\begin{itemize}
    \item \textbf{Positions and Distance Matrix} A finite set $P
\subseteq \mathbb{N}$ of \textit{positions}. A time distance matrix $D=(d_{pq})\in \mathbb{R}^{(P\times P)}$ where $d_{pq}$ represents the time needed for an driver to go from position $p$ to $q$ when not actively driving a bus. If no transfer is possible, then we set $d_{pq}$ to a big constant $M$. If $p\neq q$, then $d_{pq}$ is called {\em passive ride time}, whereas $d_{pq}$ represents the time it takes to switch tour at the same position, but is not considered passive ride time. 
    
    \item \textbf{Start and End Work}: For each position $p\in P$, two values $\param{startWork}_p$ and $\param{endWork}_p$ represent respectively the time required to start or end a shift at that position. 
   
    \item \textbf{Bus Legs}: A set $L$ of \textit{bus legs}, where  each leg  $\ell\in L$ is a $5$-tuple:
\[
 \ell =(\param{tour}_\ell,\param{startPos}_\ell, \param{endPos}_\ell, \param{start}_\ell, \param{end}_\ell),
\]
representing the trip of a vehicle between two stops at a certain time: 
\begin{itemize}
    \item $\param{tour}_\ell$ is the ID of the vehicle
    \item $\param{startPos}_\ell, \param{endPos}_\ell\in P$ are respectively the starting and the ending positions of the leg 
    \item $\param{start}_\ell\in \mathbb{R}$ is the time at which the vehicle departs from position $\param{startPos}_\ell$ 
    \item $\param{end}_\ell\in \mathbb{R}$  is the time at which the vehicle arrives to position $\param{endPos}_\ell$
\end{itemize}

\noindent Legs with the same tour $t$ do not overlap, which means that the intervals $(\param{start}_\ell, \param{end}_\ell) $ for $\ell$ with $\param{tour}_\ell = t$ are disjoint.

\end{itemize}

\noindent Note that $L$ is totally ordered by $\param{start}$, using $\param{tour}$ as tie-breaker. Therefore, the following result holds.

\begin{lemma}\label{lemma:total_order}
Let $\preceq$ be the relation on $L$ defined as follows: for any $\ell_1,\ell_2 \in L$, write $\ell_1\preceq \ell_2$ if either 
\begin{itemize}
    \item $\param{start}_{\ell_1} <\param{start}_{\ell_2} $, or 
    \item $\param{start}_{\ell_1} = \param{start}_{\ell_2}$ and $ \param{tour}_{\ell_1}  \le \param{tour}_{\ell_2}$
\end{itemize}
Then  $\preceq$ is an order relation, and  $(L, \preceq)$ is a totally ordered set.
\end{lemma}
\begin{proof}We first note that in the dataset, each pair $(\param{start}_\ell,\param{tour}_\ell)$ uniquely identifies a bus leg. That is, if two legs share the same start time and tour index, they are the same element of $L$.
\begin{itemize}
        \item \textit{Reflexivity}: For each $\ell \in L$ we have $\param{start}_{\ell}=\param{start}_{\ell}$ and $\param{tour}_{\ell}  = \param{tour}_{\ell}$. So $\ell \preceq \ell $.
   \item  \textit{Antisymmetry} Let $\ell_1,\ell_2 \in L$ be such that $\ell_1\preceq \ell_2$ and $\ell_2\preceq \ell_1$. This implies that $\param{start}_{\ell_1}=\param{start}_{\ell_2}$, $\param{tour}_{\ell_1}  \le \param{tour}_{\ell_2}$, and $\param{tour}_{\ell_2}  \le \param{tour}_{\ell_1}$. So $\ell_1 = \ell_2$.
   \item \textit{Transitivity} Let $\ell_1,\ell_2, \ell_3 \in L$ be such that  $\ell_1 \preceq \ell_2$, and $\ell_2 \preceq \ell_3$. Then there are four cases:
   \begin{enumerate}
       \item  $\param{start}_{\ell_1} <\param{start}_{\ell_2}<\param{start}_{\ell_3}$. Then $\param{start}_{\ell_1} < \param{start}_{\ell_3}$ and therefore $\ell_1 \preceq \ell_3$.
       \item  $\param{start}_{\ell_1} =\param{start}_{\ell_2}<\param{start}_{\ell_3}$ so $\ell_1 \preceq \ell_3$.
              \item  $\param{start}_{\ell_1} <\param{start}_{\ell_2}=\param{start}_{\ell_3}$ so $\ell_1 \preceq \ell_3$.
       
       \item $\param{start}_{\ell_1} =\param{start}_{\ell_2}=\param{start}_{\ell_3}$ and $\param{tour}_{\ell_1}\le \param{tour}_{\ell_2} \le \param{tour}_{\ell_3}$. This implies $\ell_1 \preceq \ell_3$.
   \end{enumerate}
   \item \textit{Totally ordered}: For any $\ell_1, \ell_2 \in L$, exactly one of the following holds:  
   \begin{itemize}
       \item   $\param{start}_{\ell_1} < \param{start}_{\ell_2}$,  
   \item  $\param{start}_{\ell_1} > \param{start}_{\ell_2}$, or  
   \item $\param{start}_{\ell_1} = \param{start}_{\ell_2}$
   \end{itemize}
   In the first two cases, $\ell_1 \preceq \ell_2$ or $\ell_2 \preceq \ell_1$ follows directly.  If $\param{start}_{\ell_1} = \param{start}_{\ell_2}$, then $\param{tour}_{\ell_1} \leq \param{tour}_{\ell_2}$ or $\param{tour}_{\ell_1} \geq \param{tour}_{\ell_2}$. Thus, $\ell_1 \preceq \ell_2$ or $\ell_2 \preceq \ell_1$.  
   Hence, $(L, \preceq)$ is totally ordered.
 \qedhere
    \end{itemize}
\end{proof}
\begin{table}[ht]
	\centering
	\begin{tabular}{*6{c}}
		\toprule
		$\ell$ & $\mathit{tour}_\ell$ & $\mathit{start}_\ell$ & $\mathit{end}_\ell$ & $\mathit{startPos}_\ell$ & $\mathit{endPos}_\ell$ \\
		\midrule
		1 & 1 & 400 & 495 & 0 & 1 \\
		2 & 1 & 510 & 555 & 1 & 2 \\
		3 & 1 & 560 & 502 & 2 & 1 \\
		4 & 1 & 508 & 540 & 1 & 0 \\
		\bottomrule
	\end{tabular}
    \caption{A Bus Tour Example}
	\label{tab:example}
\end{table}
%\begin{figure}%[b]
%	\centering
%	\begin{tikzpicture}
	%	\node [draw, circle, fill=white, inner sep=2pt, minimum size=4pt] (0) at (0,0) {0};
	%	\node [draw, circle, fill=white, inner sep=2pt, minimum size=4pt] (1) at (1,0) {1};
	%	\node [draw, circle, fill=white, inner sep=2pt, minimum size=4pt] (2) at (4,0) {2};
	%	\draw [<->, red] (0) -- (1);
	%	\path[->, blue, >= stealth'] (2) edge[bend right] node[left] {} (1);
	%	\path[->, blue, >= stealth'] (1) edge[bend right] node[left] {} (2);
	%	\end{tikzpicture}
%	\caption{Bus Movement Example}
%	\label{fig:Tour1}
%\end{figure}

Table~\ref{tab:example} shows a short example of one particular bus tour. The vehicle starts at time $400$ (6:40 AM) at position $0$, does multiple legs between positions $1$ and $2$ with waiting times in between and finally returns to position $0$. A valid tour never has overlapping bus legs and consecutive bus legs satisfy $\mathit{endPos}_i=\mathit{startPos}_{i+1}$. A tour change occurs when a driver has an assignment of two consecutive bus legs $i$ and $j$ with $\mathit{tour}_i\neq\mathit{tour}_j$. %\lucas{repetition:}
%A time distance matrix specifies, for each pair of positions $p$ and $q$, the time $d_{p,q}$ a driver takes to get from $p$ to $q$ when not actively driving a bus. If no transfer is possible, then $d_{p,q}=\infty$. $d_{p,q}$ with $p\neq q$ is called the {\em passive ride time}. $d_{p,p}$ represents the time it takes to switch tour at the same position, but is not considered passive ride time.
%Finally, each position $p$ is associated with an amount of working time for starting a shift  ($\mathit{startWork}_p$) and ending a shift ($\mathit{endWork}_p$) at that position. 
\subsection{Solution}
A solution $S$ to the \gls{bdsp} is an assignment of drivers to bus legs. 	Formally, it is represented as a set  partition of $L$, expressed as $S =\{s_1,s_2,\dots, s_n\}$.  Each block $s_i$, is called a \textbf{shift}. In practical terms, a \textit{shift} corresponds to the work scheduled to be performed by a driver in one day~\cite{wren2004scheduling}.

For a given shifts $s$, we denote $L_s\subseteq  L$ as the subset of bus legs assigned to $s$. While $s$ and $L_s$ are mathematically equivalent, we conceptually distinguish them:  $L_s$ represents only a set of bus legs, whereas $s$ also contains additional details such as breaks, passive ride time, etc\dots. Those additional information are uniquely determined by the set of legs $L_s$, this means that, if necessary, the reader can think about $s$ and $L_s$ interchangeably.

From Lemma~\ref{lemma:total_order}, it follows that the notion of  \textit{consecutive} bus legs in a shift is well-defined.
 Note that a priori, the number of shifts of a solution $\abs{S}$ is not given. Nevertheless, we can imagine to set it as large as we need in order to get a feasible solution. An immediate upper bound is $\abs{S} \le \abs{L}$. This represents the situation in which each shift is composed by only one leg. 

Each shift $s\in S$ must be feasible according to the following criteria:
\begin{itemize}
	\item No overlapping bus legs are assigned to $s$.
	\item Changing tour or position between consecutive legs $i,j\in s$ requires \[\mathit{start}_j\ge\mathit{end}_i+d_{\mathit{endPos}_i,\mathit{startPos}_j}.\]
	\item The shift $s$ respects all hard constraints regarding work regulations as specified in the next section.
\end{itemize}
\subsection{Work and Break Regulations}
\begin{figure}[htp]
	\centering
	\begin{tikzpicture}
		\draw (0,-.1) rectangle (.5,.1);
		\draw[dashed] (.25,.2) -- (.25,.5) node[above] {\small start work};
		\draw (.5,-.25) rectangle (1.5,.25) node[pos=.5] {$\ell_1$};
		\draw (1.5,0) -- (2.5,0);
		\draw[dashed] (2,-.1) -- (2,-.4) node[below] {\small rest};
		\draw (2.5,-.25) rectangle (4,.25) node[pos=.5] {$\ell_2$};
		\draw (4,0) -- (4.5,0);
		\draw[dashed] (4.25,-.1) -- (4.25,-.4) node[below] {\small rest};
		\draw (4.5,-.1) rectangle (5.4,.1);
		\draw[dashed] (4.95,.2) -- (4.95,.5) node[above] {\small passive ride};
		\draw (5.4,-.25) rectangle (7,.25) node[pos=.5] {$\ell_3$};
		\draw (7,-.1) rectangle (7.25,.1);
		\draw[dashed] (7.125,.2) -- (7.125,.5) node[above] {\small end work};
		\node (wt) at (3.625,-1.5) {\small Working time $W_s$};
		\draw[dashed] (.25,-.2) edge[bend right=40] (wt.west);
		\draw[dashed] (1,-.3) edge[bend right=30] (wt.west);
		\draw[dotted] (2,-.9) edge[bend right=10] node[right] {\small ?} (wt.north west);
		\draw[dashed] (3.25,-.3) edge[bend right=10] (wt);
		\draw[dotted] (4.25,-.9) edge[bend left=10] node[left] {\small ?} ([xshift=4mm]wt.north);
		\draw[dashed] (4.95,-.2) edge[bend left=30] (wt);
		\draw[dashed] (6.2,-.3) edge[bend left=30] (wt.east);
		\draw[dashed] (7.125,-.2) edge[bend left=40] (wt.east);
		\node (dt) at (3.625,1.5) {\small Driving time $D_s$};
		\draw[dashed] (1,.3) edge[bend left=30] (dt.west);
		\draw[dashed] (3.25,.3) edge[bend left=10] (dt);
		\draw[dashed] (6.2,.3) edge[bend right=30] (dt.east);
		\draw[dashed] (0,-.2) -- (0,-2);
		\draw[dashed] (7.25,-.2) -- (7.25,-2);
		\draw[latex'-latex'] (0,-2) -- (7.25,-2) node[pos=.5,below] {\small Total time $T_s$};
	\end{tikzpicture}
	\caption{Example shift~\cite{kletzander2020solving}}
	\label{fig:shiftExample}
\end{figure}
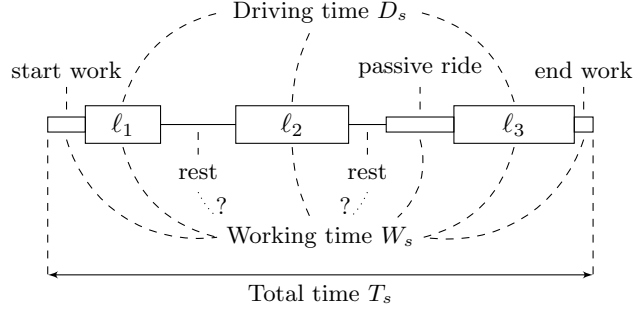
Valid shifts for drivers are constrained by work regulations and require frequent breaks. First, different measures of time related to an employee $e$ containing the set of bus legs $L_s$ need to be distinguished, as visualised in Figure~\ref{fig:shiftExample}:
\begin{itemize}
	\item The total amount of driving time: $D_s=\sum_{i\in L_s}{\mathit{drive}_i}$.
	\item The span from the start of work until the end of work $T_s$ with a maximum of $T_\text{max}=\qty{14}{\hour}$.
	\item The working time $W_s=T_s-\mathit{unpaid}_s$, which does not include certain unpaid breaks.
\end{itemize}
\subsubsection{Driving Time Regulations.}
The maximum driving time is restricted to $D_\text{max}=\qty{9}{\hour}$. The whole distance $\mathit{start}_j-\mathit{end}_i$ between consecutive bus legs $i$ and $j$ qualifies as a driving break, including passive ride time. Breaks from driving need to be taken repeatedly after at most \qty{4}{\hour} of driving time. In case a break is split in several parts, all parts must occur before a driving block exceeds the \qty{4}{\hour} limit. Once the required amount of break time is reached, a new driving block starts. The following options are possible:
\begin{itemize}
	\item One break of at least \qty{30}{\minute};
	\item Two breaks of at least \qty{20}{\minute} each;
	\item Three breaks of at least \qty{15}{\minute} each.
\end{itemize}

\subsubsection{Working Time Regulations.}

The working time $W_s$ cannot exceed \qty{10}{\hour} and has a soft minimum of \qty{6.5}{\hour}. If the employee is working for a shorter period of time, the difference has to be paid anyway. 

A minimum rest break is required according to the following options:
\begin{itemize}
	\item $W_s<\qty{6}{\hour}$: no rest break;
	\item $\qty{6}{\hour}\leq W_s\leq\qty{9}{\hour}$: at least \qty{30}{\minute};
	\item $W_s>\qty{9}{\hour}$: at least \qty{45}{\minute}.
\end{itemize}

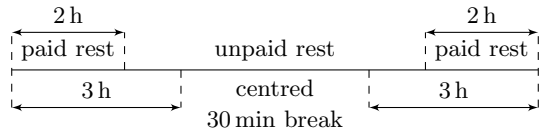
\begin{figure}[htp]%[t]
	\centering
	\begin{tikzpicture}
		\draw (0,0) -- (7,0) node[pos=.5,above] {\small unpaid rest};
		\draw[dashed] (0,-.5) -- (0,.5);
		\draw[dashed] (1.5,0) -- (1.5,.5);
		\draw[dashed] (2.25,0) -- (2.25,-.5);
		\draw[dashed] (4.75,0) -- (4.75,-.5);
		\draw[dashed] (5.5,0) -- (5.5,.5);
		\draw[dashed] (7,-.5) -- (7,.5);
		\draw[latex'-latex'] (0,.5) -- (1.5,.5) node[pos=.5,above] {\small \qty{2}{\hour}};
		\draw[latex'-latex'] (5.5,.5) -- (7,.5) node[pos=.5,above] {\small \qty{2}{\hour}};
		\draw[latex'-latex'] (0,-.5) -- (2.25,-.5) node[pos=.5,above] {\small \qty{3}{\hour}};
		\draw[latex'-latex'] (4.75,-.5) -- (7,-.5) node[pos=.5,above] {\small \qty{3}{\hour}};
		\node[above] at (.75,0) {\small paid rest};
		\node[above] at (6.25,0) {\small paid rest};
		\node[below,align=center] at (3.5,0) {\small centred\\\small \qty{30}{\minute} break};
	\end{tikzpicture}
	\caption{Rest break positioning~\cite{kletzander2020solving}}
	\label{fig:rest}
\end{figure}
\noindent
The rest break may be split into one part of at least \qty{30}{\minute} and one or more parts of at least \qty{15}{\minute}. The first part has to occur after at most \qty{6}{\hour} of working time.
Whether rest breaks are paid or unpaid depends on break positions according to Figure~\ref{fig:rest}. Every period of at least \qty{15}{\minute} of consecutive rest break is unpaid as long as it does not intersect the first 2 or the last 2 hours of the shift (a longer rest break might be partially paid and partially unpaid). The maximum amount of unpaid rest is limited: 
% If 30 consecutive minutes of rest break are located such that they do not intersect the first \qty{3}{\hour} of the shift or the last \qty{3}{\hour} of the shift, at most \qty{1.5}{\hour} of unpaid rest are allowed. Otherwise, at most one hour of unpaid rest is allowed. Rest breaks beyond this limit are paid.
\begin{itemize}
	\item If 30 consecutive minutes of rest break are located such that they do not intersect the first \qty{3}{\hour} of the shift or the last \qty{3}{\hour} of the shift, at most \qty{1.5}{\hour} of unpaid rest are allowed;
	\item Otherwise, at most one hour of unpaid rest is allowed.
\end{itemize}
\noindent

\subsubsection{Shift split}
If a rest break is at least \qty{3}{\hour} long, it is instead considered a shift split, which is unpaid and does not count towards $W_s$. However, such splits are typically poorly regarded by the drivers. A shift split counts as a driving break, but does not contribute to rest breaks.
\subsection{Objective Function}
We minimize the objective function combining cost and employee satisfaction defined in previous work~\cite{kletzander2020solving}:
\begin{equation}
	\label{eq:objective}
	z(S)=\sum_{s \in S}\left(2\, W_s'+T_s+\mathit{ride}_s+30\, \mathit{change}_s+180\, \mathit{split}_s\right)
\end{equation}
\noindent
The objective function $z$ represents a linear combination of six criteria for each employee $e$. The actual paid working time $W_s'=\max\{W_s, 390\}$ is the main objective, and it is combined with the total time $T_s$ to reduce long unpaid periods for employees. The next sub-objectives reduce the passive ride time $\mathit{ride}_s$ and the number of tours changes $\mathit{change}_s$, which is beneficial for both employees and efficient schedules. The last objective aims to reduce the number of split shifts $\mathit{split}_s$ as they are very unpopular.
The weights were determined by previous work~\cite{kletzander2020solving} based on preferences agreed by different stakeholders at Austrian bus companies and employee scheduling experts.

\section{Branch and Price}
\label{sec:bp}

\begin{algorithm}[t]
    \caption{Branch and Price}
    \label{alg:bp}
    \KwIn{Problem instance}
    \KwOut{Best solution (optimal if no timeout occurs)}

    $columns\gets$ initialize columns\;
    $best\gets\infty$\;
    $nodes \gets \{\mathrm{root}\}$\;
    \While{nodes is not empty and no timeout}{
        $node, columns\gets\mathrm{choose}(nodes,columns)$\;
        \If{$node_{lb}\geq best$}{
            \textbf{continue}\;
        }
        \While{no timeout}{
            $rmp,~duals\gets$ solve relaxed master problem($columns$)\;
            $new\_columns\gets$ solve subproblem($duals$)\;
            \If{new\_columns is empty}{
                \textbf{break}\;
            }
            $columns\gets columns\cup new\_columns$\;
        }
        \If{timeout}{
            \If{$best=\infty$}{
                $best\gets$ solve integer master problem($columns$)\;
            }
            \Return{$best$}\;
        }
        \eIf{rmp is integer}{
            $best\gets\min(rmp,best)$\;
        }{
            $best\gets\min(\text{solve integer master problem}(columns),best)$\;
            $nodes\gets nodes\cup{}$branch($rmp$)\;
        }
    }
    \Return{best}\;
\end{algorithm}

An exact solution to the \gls{bdsp} can be computed by \acrfull{bp} \cite{barnhart1998branch}. The general procedure is shown in Algorithm~\ref{alg:bp}. The outer loop (lines 4 to 29) iterates over the nodes in the branching tree. The inner loop (lines 9 to 16) performs \acrfull{cg} for each node. Set Partitioning is used as the master problem (line 10) and the \acrfull{rcspp} as the subproblem (line 11). The duals of the relaxed master problem are used to find new shifts that have the potential to improve the solution of the master problem. This is repeated until no columns with potential for improvement are found (lines 12 to 14). Branching occurs when the resulting solution is not integer (line 27). The process continues until all branches either result in integer solutions (lines 23 to 24) or are cut off by the current objective bounds (lines 6 to 8).

\subsection{Master Problem}
\label{sec:master}

The goal of the master problem is to select a subset of the shift set $\mathbf{S}$, such that each bus leg is covered by exactly one shift while minimizing total cost. This corresponds to the set partitioning problem:

\begin{align}
\label{eq:sp_cost}
\text{minimize } &\sum_{s\in\mathbf{S}}{\mathit{cost}_s\cdot x_s}\\
\label{eq:sp_cover}
\text{subject to } &\sum_{s\in\mathbf{S}}{\mathit{cover}_{s\ell}\cdot x_s}=1 &\forall\ell\\
\label{eq:sp_int}
& x_s\in\{0,1\} &\forall s
\end{align}

\noindent
Here $x_s$ is the variable for the selection of shift $s$. The objective \eqref{eq:sp_cost} minimizes the total cost, Equation~\eqref{eq:sp_cover} states that each bus leg needs to be covered exactly once (using $\mathit{cover}_{s\ell}\in\{0,1\}$ to indicate whether shift $s$ covers leg $\ell$), and Equation~\eqref{eq:sp_int} states the integrality constraint. This constraint is relaxed to $0\leq x_s\leq1$ for the relaxed master problem which is repeatedly solved at each node of the branching tree (line 10 in Algorithm~\ref{alg:bp}). Instead of the full set of possible shifts $\mathbf{S}$, a subset $columns$ is maintained by the algorithm. Once no more new shifts can be found by the subproblem, the result of the relaxed master problem provides a local lower bound for the solution of the integer problem.

In some cases, the resulting solution might already be integer (line 23). Usually, however, the result will be fractional. Then, the integer version of the master problem is solved with the current set of columns (line 26). While there are no guarantees regarding the quality of the integer solution in this case, in practice, solutions are often very close to the lower bound already. In any case, the result from the integer master problem is a feasible solution that provides a global upper bound for the problem. If the solution is not integer, branching will be done, resuming calculation on one of the open branches.

\subsection{Subproblem}

The goal of the subproblem (line 11 in Algorithm~\ref{alg:bp}) is to find the column (shift) with the lowest reduced cost. For the \gls{bdsp} this corresponds to the \acrfull{rcspp} on an acyclic graph. In this problem, a graph $G=(\mathbf{N},\mathbf{A})$ is given where $\mathbf{N}$ is the set of $n$ nodes, corresponding to all bus legs, a source node $s$, and a target node $t$, and $\mathbf{A}$ is the set of arcs, which connect bus legs that can be scheduled consecutively in the same shift. As the bus legs are naturally ordered by time, the \gls{rcspp} is acyclic.

Each node and arc is associated with a cost and an $r$-dimensional resource vector representing the resource consumption when using the node or arc. A shift is defined as a path from $s$ to $t$ such that the path satisfies all feasibility criteria associated with the resources. Each node corresponding to a bus leg is associated with a dual from the master problem. The reduced cost of a path is therefore the cost of a path minus the sum of the duals along the path. The \gls{rcspp} aims at finding the least-cost feasible path from $s$ to $t$.

There is an arc from node $s$ to each node $i$ except $i=t$, and there is an arc from each node $i$ except $i=s$ to node $t$. Nodes corresponding to bus legs $i$ and $j$ are only connected by an arc $ij$ if chaining the bus legs is feasible according to their times and the distance matrix $d$. Building the graph is done once per instance in $O(n^2)$. Figure~\ref{fig:rcspp} shows such a graph for a small toy instance, skipping most arcs from $s$ and to $t$ for visual clarity. The remainder of this section goes into the details of this graph, its costs, and its constraints. Due to the complex regulations, solving this sub-problem is very challenging and requires several novel optimizations.

\begin{figure}[htp]
\centering
\begin{tikzpicture}[node distance=2 cm, auto]
\node[draw,circle] at (0,0) (s) {$s$};
\node[draw,circle,label=above:{$63$}] at (1,1.5) (v1) {$0$};
\node[draw,circle,label=above:{$111$}] at (2,1.5) (v2) {$1$};
\node[draw,circle,label=below:{$111$}] at (2.5,0) (v3) {$2$};
\node[draw,circle,label=above:{$90$}] at (3,-1.5) (v4) {$3$};
\node[draw,circle,label=above:{$87$}] at (3.5,1.5) (v5) {$4$};
\node[draw,circle,label=above:{$144$}] at (4,-1.5) (v6) {$5$};
\node[draw,circle,label=below:{$105$}] at (4.5,0) (v7) {$6$};
\node[draw,circle,label=above:{$84$}] at (5,1.5) (v8) {$7$};
\node[draw,circle,label=above:{$60$}] at (5.5,-1.5) (v9) {$8$};
\node[draw,circle,label=below:{$105$}] at (6,0) (v10) {$9$};
\node[draw,circle] at (7.5,0) (t) {$t$};

\path[line] (s) edge node {$45$} (v1);
\path[line] (v1) edge node {$78$} (v2);
\path[line] (s) edge node {$45$} (v3);
\path[line] (s) edge node[swap] {$45$} (v4);
\path[line] (v2) edge node {$57$} (v5);
\path[line] (v1) edge[bend left=15] node[pos=.25,left] {$394$} (v6);
\path[line] (v4) edge node[swap] {$18$} (v6);
\path[line] (v3) edge node {$72$} (v7);
\path[line] (v5) edge node {$78$} (v8);
\path[line] (v4) edge[bend right=60] node[swap] {$255$} (v9);
\path[line] (v6) edge node[swap] {$42$} (v9);
\path[line] (v2) edge node[pos=.3,below] {$489$} (v10);
\path[line] (v5) edge node[pos=.75,above] {$296$} (v10);
\path[line] (v7) edge node[swap] {$72$} (v10);
\path[line] (v8) edge node {$30$} (t);
\path[line] (v9) edge node[swap] {$30$} (t);
\path[line] (v10) edge node[swap] {$30$} (t);
\end{tikzpicture}
\caption{RCSPP graph for a Toy Instance}
\label{fig:rcspp}
\end{figure}
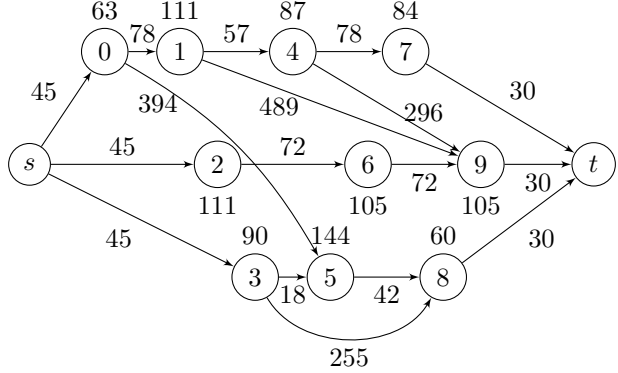

\subsubsection{Costs}
The costs for nodes and arcs are based on objective \eqref{eq:objective}. The cost $c_i$ for each node $i$ corresponding to a bus leg $i$ is defined as $c_i=3\cdot\mathit{length}_i$ as each bus leg contributes its length to both $W_s$ (weight 2) and $T_s$ (weight 1). By definition, $c_s=c_t=0$. 

Several helpful properties of arcs from node $i$ to node $j$ are defined and used for defining the arc costs:
{\small
\begin{align}
    \label{eq:length}
    \mathit{length}_{ij} &= \begin{cases}\mathit{start}_j-\mathit{end}_i & \text{if }i\neq s\land j\neq t\\
    \mathit{startWork}_{\mathit{startPos}_j} & \text{if }i=s\\
    \mathit{endWork}_{\mathit{endPos}_i} & \text{if }j=t
    \end{cases}
    \\
    \label{eq:ride}
    \mathit{ride}_{ij} &= \begin{cases}d_{\mathit{endPos}_i,\mathit{startPos}_j} & \text{if }i\neq s\land j\neq t\land\mathit{endPos}_i\neq\mathit{startPos}_j\\
    0 & \text{otherwise}
    \end{cases}
    \\
    \label{eq:change}
    \mathit{change}_{ij} &= \begin{cases}1 & \text{if }i\neq s\land j\neq t\land\mathit{tour}_i\neq\mathit{tour}_j\\
    0 & \text{otherwise}
    \end{cases}
    \\
    \label{eq:split}
    \mathit{split}_{ij} &= \begin{cases}1 & \text{if }i\neq s\land j\neq t\land\mathit{length}_{ij}-\mathit{ride}_{ij}\geq3\cdot60\\
    0 & \text{otherwise}
    \end{cases}
    \\
    \label{eq:remain}
    \mathit{remain}_{ij} &= \begin{cases}\mathit{length}_{ij}-\mathit{ride}_{ij} & \text{if }\mathit{split}_{ij}=0\\
    0 & \text{otherwise}
    \end{cases}
    \\
    \label{eq:rest}
    \mathit{rest}_{ij} &= \begin{cases}\mathit{remain}_{ij} & \text{if }i\neq s\land j\neq t\land\mathit{remain}_{ij}\geq15\\
    0 & \text{otherwise}
    \end{cases}
\end{align}
}%

\noindent
Equation~\eqref{eq:length} defines the length of an arc, taking into account start and end arcs. Equation~\eqref{eq:ride} states the passive ride time, Equation~\eqref{eq:change} whether a tour change occurs, Equation~\eqref{eq:split} whether the arc corresponds to a shift split, Equation~\eqref{eq:remain} captures the remaining arc length after removing the passive ride time and the shift split time, and Equation~\eqref{eq:rest} captures a potential rest break.

The cost $c_{ij}$ for arc $ij$ (i.e., bus leg $i$ to $j$) is defined as
\begin{equation}
    \label{eq:arc_cost}
    c_{ij}=2\cdot\mathit{remain}_{ij}+\mathit{length}_{ij}+3\cdot\mathit{ride}_{ij}+30\cdot\mathit{change}_{ij}+180\cdot\mathit{split}_{ij}
\end{equation}
\noindent
Note that unpaid rest cannot be determined at this point, therefore all rest is treated as paid for this computation. Unpaid rest is separately treated when solving the subproblem. $\mathit{ride}_{ij}$ contributes to both working time $W_s$ and the additional passive ride penalty.

\subsubsection{Resources}

The solution method is a label-setting algorithm \cite{irnich2005shortest}. Due to the acyclic nature of the graph, each node can be processed in temporal order, starting with an initial label representing an empty path at $s$ where each resource usage is 0. For each node $i$ and each label $x$ on that node, all arcs $ij$ are processed and a label $y$ is placed at the end node $j$ of the arc unless some constraints are violated. Therefore, each label $x$ represents a path from $s$ to its node $i$, capturing the nodes in the path, the cost, and the resource usage. In total, eleven resources need to be tracked in order to satisfy the BDS constraints. The first resources are classical additive resources with a maximum allowed usage.
\begin{align}
    \label{eq:dy}
    d_y &= d_x+\mathit{drive}_j\\
    \label{eq:sy}
    s_y &= s_x+\mathit{length}_{ij}+\mathit{length}_j
\end{align}

\noindent
Equation~\eqref{eq:dy} tracks driving time, Equation~\eqref{eq:sy} the span.
{\small
\begin{align}
    \label{eq:rdy}
    \mathit{rd}_y &= \begin{cases}\mathit{true} & \text{if }\mathit{length}_{ij}\geq30\lor (\mathit{length}_{ij}\geq20\land\mathit{b20}_x\geq1)\lor (\mathit{length}_{ij}\geq15\land\mathit{b15}_x\geq2)\\
    \mathit{false} & \text{otherwise}
    \end{cases}
    \\
    \label{eq:dcy}
    \mathit{dc}_y &= \begin{cases}0 & \text{if }\mathit{rd}_y\\
    \mathit{dc}_x+\mathit{drive}_j & \text{otherwise}
    \end{cases}
    \\
    \label{eq:b15y}
    \mathit{b15}_y &= \begin{cases}0 & \text{if }\mathit{rd}_y\\
    \mathit{b15}_x+1 & \text{if }\neg\mathit{rd}_y\land\mathit{length}_{ij}\geq15\\
    \mathit{b15}_x & \text{otherwise}
    \end{cases}
    \\
    \label{eq:b20y}
    \mathit{b20}_y &= \begin{cases}0 & \text{if }\mathit{rd}_y\\
    \mathit{b20}_x+1 & \text{if }\neg\mathit{rd}_y\land\mathit{length}_{ij}\geq20\\
    \mathit{b20}_x & \text{otherwise}
    \end{cases}
\end{align}
}%

\noindent
Resources for monitoring drive breaks need to be reset at each full drive break. Equation~\eqref{eq:rdy} defines a helping flag to indicate whether the current drive block is finished. Equation~\eqref{eq:dcy} uses this flag to reset or increase the current driving time. Equation \eqref{eq:b15y} tracks the number of 15-minute breaks in the current driving block and Equation~\eqref{eq:b20y} the number of 20-minute breaks. 30-minute breaks do not need to be tracked as each of them resets the driving block.

Constraints regarding rest breaks need to consider different sums of rest breaks and their positioning.
{\small
\begin{align}
    \label{eq:wy}
    w_y &= w_x+u_x-u_y+\mathit{remain}_{ij}+\mathit{ride}_{ij}+\mathit{length}_j\\
    \label{eq:ry}
    r_y &= \min(r_x+\mathit{rest}_{ij},~45)\\
    \label{eq:b30y}
    \mathit{b30}_y &= \mathit{b30}_x\lor\mathit{rest}_{ij}\geq30\\
    \label{eq:restp}
    \mathit{rest}_{ij}' &= \mathit{rest}_{ij}-\max(2\cdot60-s_x,~0)-\max(\mathit{end}_i+\mathit{rest}_{ij}-(\mathit{end}_y-2\cdot60),~0)\\
    \label{eq:uy}
    u_y &= \min\left(u_x+\begin{cases}\mathit{rest}_{ij}' & \text{if }\mathit{rest}_{ij}'\geq15\\
    0 & \text{otherwise}
    \end{cases},~90\right)\\
    \label{eq:bc30y}
    \mathit{bc30}_y &= \mathit{bc30}_x\lor\mathit{rest}_{ij}''\geq30
\end{align}
}%

\noindent
Equation~\eqref{eq:wy} tracks working time, assuming that all of $u_y$ is unpaid so far, which constitutes a lower bound for the actual value. Equation~\eqref{eq:ry} tracks the amount of required rest time, capping it at 45 as higher values do not matter. Equation~\eqref{eq:b30y} deals with the occurrence of a 30-minute rest break.

\begin{figure}[htp]
\centering
\begin{subcaptionblock}{.45\textwidth}
\begin{tikzpicture}[node distance=2 cm, auto]
\node[draw,circle] at (0,0) (s) {$s$};
\node[draw,circle] at (1,.5) (v1) {$0$};
\node[draw,circle] at (2,-.5) (v2) {$1$};
\node[draw,circle] at (3,.5) (v3) {$2$};
\node[draw,circle] at (4,.5) (v4) {$3$};
\node[draw,circle] at (5,0) (t) {$t$};

\path[line] (s) edge (v1);
\path[line] (s) edge (v2);
\path[line] (s) edge[bend left=50] (v3);
\path[line] (s) edge[bend left=60] (v4);
\path[line] (v1) edge (v2);
\path[line] (v1) edge (v3);
\path[line] (v2) edge (v4);
\path[line] (v3) edge (v4);
\path[line] (v1) edge[bend left=60] (t);
\path[line] (v2) edge (t);
\path[line] (v3) edge[bend left=55] (t);
\path[line] (v4) edge (t);
\end{tikzpicture}
\caption{Original graph}
\end{subcaptionblock}
\begin{subcaptionblock}{.45\textwidth}
\begin{tikzpicture}[node distance=2 cm, auto]
\node[draw,circle] at (0,0) (s) {$s$};
\node[draw,circle,inner sep=2pt] at (1,2) (v00) {$0_0$};
\node[draw,circle] at (1,1) (v0) {$0$};
\node[draw,circle] at (2,.5) (v1) {$1$};
\node[draw,circle,inner sep=2pt] at (3,1) (v23) {$2_3$};
\node[draw,circle,inner sep=2pt] at (4,1) (v33) {$3_3$};
\node[draw,circle,inner sep=2pt] at (3,-.5) (v22) {$2_2$};
\node[draw,circle,inner sep=2pt] at (1,-1) (v01) {$0_1$};
\node[draw,circle,inner sep=2pt] at (2,-1.5) (v11) {$1_1$};
\node[draw,circle] at (5,0) (t) {$t$};

\path[line] (s) edge[bend left=20] (v00);
\path[line] (v00) edge[bend left=60] (t);
\path[line] (s) edge (v0);
\path[line] (s) edge (v1);
\path[line] (v0) edge (v1);
\path[line] (s) edge[bend left=65] (v23);
\path[line] (s) edge[bend right=20] (v33);
\path[line] (v0) edge (v23);
\path[line] (v1) edge[bend right=20] (v33);
\path[line] (v23) edge (v33);
\path[line] (v33) edge (t);
\path[line] (s) edge (v22);
\path[line] (v0) edge[bend right=20] (v22);
\path[line] (v22) edge (t);
\path[line] (s) edge (v01);
\path[line] (s) edge[bend left=15] (v11);
\path[line] (v01) edge (v11);
\path[line] (v11) edge (t);
\end{tikzpicture}
\caption{Extended graph}
\end{subcaptionblock}
\caption{Extension of \gls{rcspp} graph to deal with distance to shift end}
\label{fig:rcspp_extension}
\end{figure}

Equation~\eqref{eq:restp} captures the part of the rest break that can be unpaid depending on the position. However, this requires knowing the end time of the shift $\mathit{end}_y$ which violates the general assumption of the algorithm that nodes can be processed in temporary order. Therefore, for each node $j$, with known end time $\mathit{end}_j$ if $j$ is the last bus leg in the shift, new nodes are added to the graph: $\mathbf{N}_j$ is the set of nodes reachable within 3 hours from $\mathit{end}_j$ when traversing the network backwards from $j$ (including $j$ itself). For each node $i\in\mathbf{N}_j$ a new node $i_j$ is created. For each pair of nodes $i$ and $k$ both in $\mathbf{N}_j$, an arc $i_jk_j$ is added if $ik$ is in the original graph. For $i\not\in\mathbf{N}_j$ and $k\in\mathbf{N}_j$, an arc $ik_j$ is added if $ik$ is in the original graph. The arc $j_jt$ replaces the arc $jt$. Each new node is associated with the end time $\mathit{end}_j$. All original nodes have end time $\mathit{\infty}$. This solves the problem, but increases the graph size by a factor of 6 for small instances up to more than 30 for large instances.

Figure~\ref{fig:rcspp_extension} shows this extension on a very small graph, assuming that only node $n-1$ is in the 3 hour window ending at the end of node $n$. The extended graph contains sub-networks for each final node, e.g., at node $s$ going to node $0_1$ locks into a path ending with node $1_1$, while going to $0$ keeps the choice open to end via $2_2$ or $3_3$, but not directly via $0_0$ or $1_1$ any more.

The potential total amount of unpaid rest is captured in Equation~\eqref{eq:uy}. Equation~\eqref{eq:bc30y} tracks the existence of a centered 30-minute break using $\mathit{rest}_{ij}''$ defined like in~\eqref{eq:restp} except with $3\cdot60$ instead of $2\cdot60$.

Finally, the cost needs to be computed.
\begin{align}
    \label{eq:costp}
    \mathit{cost}_y' &= \mathit{cost}_x'+2\cdot(u_x-u_y)+c_{ij}+c_j-\mathit{dual}_j\\
    \label{eq:cost}
    \mathit{cost}_y &= \mathit{cost}_y'+2\cdot\max(W_{min}-w_i,~0)
\end{align}

\noindent
Equation~\eqref{eq:costp} takes into account the unpaid rest and the duals for the bus legs. Equation~\eqref{eq:cost} further considers the minimum working time.

\subsubsection{Constraints}
A shift is a path from $s$ to $t$ that does not violate any resource constraints. The following resource constraints need to be satisfied for every label $x$: $d_x\leq D_{max}$, $s_x\leq T_{max}$, $\mathit{dc}_x\leq4\cdot60$, $w_x\leq W_{max}$, and $r_x\geq15$ if $w_x\geq6\cdot60$. Additionally, at node $t$, all labels $x$ need to satisfy: $r_x\geq45$ if $w_x>9\cdot60$ and $\mathit{b30}_x$ if $w_x\geq6\cdot60$.

\subsubsection{Dominance}

In order to track only paths with the potential to become minimum-cost paths, only pareto-optimal labels with respect to both cost and resource usage are kept at each node, i.e., labels which are not dominated. Regarding more complex constraints, it is important to observe that, if label $x$ dominates label $y$ at a node $i$, extensions of $x$ to future nodes must still dominate extensions of $y$. Finally, the least-cost label at node $t$ represents the minimum-cost path.

Therefore, a label $x$ dominates label $y$ if all following conditions are true: $d_x\leq d_y$, $s_x\leq s_y$, $\mathit{dc}_x\leq\mathit{dc}_y$, $\mathit{b20}_x\geq\mathit{b20}_y$, $\mathit{b15}_x\geq\mathit{b15}_y$, $w_x+u_x\leq w_y$, $r_x\geq r_y$, $\mathit{b30}_x\lor\neg\mathit{b30}_y$, $\mathit{bc30}_x\lor\neg\mathit{bc30}_y$, $c_x+2\cdot u_x\leq c_y$. 

In general, lower values dominate for resources with upper bounds and higher values dominate for resources with lower bounds. The most complex case arises from the fact that the maximum value of unpaid break might be 0, 60, or 90 depending on the existence and position of a 30-minute rest break. This results in neither higher nor lower values of $u_i$ being dominant. Rather, the uncertain amount of unpaid rest weakens the domination power of cost and working time, as those might vary in the amount of unpaid rest until the end of the path. However, in the next section a more useful way to deal with this problem is described.

\subsection{Branching}

Branching (line 27 in Algorithm~\ref{alg:bp}) is performed on the connections between bus legs in a shift that correspond to arcs in the subproblem. The branching considers a connection between bus legs $i$ and $j$ that appears fractionally in the solution to the relaxed master problem, and selects the most fractional (closest to $0.5$) to maximize impact. In the left branch, all columns containing $i$ or $j$, but not both consecutively, are removed. In the RCSPP graph all outgoing arcs from $i$ except for the one connecting to $j$ are removed. In the right branch, all columns with $i$ and $j$ assigned consecutively are removed, as well as the arc from $i$ to $j$ in the RCSPP graph.

\subsection{Lagrangean Bound}

In the case when the column generation terminates at the root node, small optimality gaps are typically achieved. However, when there is not enough time to complete the column generation at the root node, a Lagrangean lower bound can be computed as $O(\mathit{RMP})+\kappa\cdot O(\mathit{PP})$, where $O(\mathit{RMP})$ is the optimum of the reduced master problem, $\kappa$ is an upper bound on the number of shifts and $O(PP)$ is the optimum of the pricing problem. $\kappa$ can be bounded by $\kappa=\lfloor O(\mathit{RMP})/\mathit{minCost}\rfloor$,
where $\mathit{minCost}=2\cdot W_{min}+\min_{\ell\in\mathbf{L}}{\mathit{length}_\ell}$. The minimum reduced cost is only known when no more throttling (see next section) is in play. Otherwise a lower bound for the minimum reduced cost can be computed where the constraints regarding break positions for unpaid breaks are relaxed. However, these bounds turned out to be rather weak for the \gls{bdsp}.

\subsection{Handling the Subproblem Dimensionality}

As described in the previous section, due to the complex constraints for valid shifts, the subproblem complexity is very high. In particular, eleven resources are tracked in the labels. However, the efficiency of the label-setting algorithm depends on its ability to keep a small number of non-dominated labels. The more dimensions the labels have, the easier it is for labels to be non-dominated, drastically increasing the number of labels for processing. Therefore, the increase in the number of processed labels turns out to be the bottleneck of scaling the solution method to larger instances. Several novel improvements to reduce this bottleneck are introduced. These are generally applicable to other problems with similar characteristics.

\subsubsection{Subproblem Partitioning}

Instead of generating all possible shifts from one graph, the subproblem is split into three similar, but disjoint RCSPP problems, removing the uncertainty for unpaid breaks depending on whether the maximum amount of rest break is 0, 60, or 90. This depends on the existence and position of a 30-minute rest break:
\begin{itemize}
\item {\em No 30-minute rest break}: All arcs corresponding to rest breaks of at least 30 minutes can be removed, making the graph very spare and therefore fast to process. Unpaid rest is guaranteed to be 0, all rest break resources are ignored.
\item {\em Uncentered 30-minute rest break}: $\mathit{b30}_y$ must be true at $t$. The maximum of $u_y$ is changed to 60, but the current amount of $u_y$ is guaranteed to be unpaid. Therefore, $w_x\leq w_y$ and $w_x+u_x\leq w_y+u_y$ can be used instead of $w_x+u_x\leq w_y$ as the domination criterion (same change for $c_x$ and $c_y$), reducing the number of undominated labels.
\item {\em Centered 30-minute rest break:} $\mathit{b30}_y$ and $\mathit{b30c}_y$ must be true at $t$. The maximum of $u_y$ is set to 90, but the improved dominance criteria from the previous graph still apply as again all of $u_y$ is guaranteed to be unpaid.
\end{itemize}

\noindent
The pricing subproblem is used to add multiple columns (up to 1000 per graph) at once. Indeed, solving the relaxed master problem even with thousands of columns is much faster than solving the pricing subproblem. Therefore, accepting some unnecessary columns while reducing the number of times the subproblem is solved pays off. The different graphs are ordered by their complexity, measured by the number labels they expand. The pricing subproblem avoids searching more complex graphs, as long as those with less complexity still produce enough shifts with negative reduced costs (usually 10, 100 if the objective did not change, 1000 if the objective did not change repeatedly).

\subsubsection{Cost Bound}

An upper bound for the reduced cost for each node can be computed by processing the RCSPP graph backwards with the current duals. The upper bound describes the maximum reduced cost at each node, such that there is still a path to arrive at $t$ with reduced cost $<0$, disregarding resource constraints. All labels above this bound can be discarded during the solution process.

\subsubsection{Exponential Arc Throttling}

While it might be hard to find the best column, at least in the beginning of the solving process, there are many columns with negative reduced cost. Two options were considered to reduce the size of the subproblem in early iterations. The first option is to reduce the number of labels at each node. This can be done either by rejecting new labels after a certain threshold or only retaining labels according to certain criteria, e.g., the best 100 by cost.

The better option is to exploit the fact that good columns are unlikely to include very costly arcs. It proposes a new throttling mechanism that imposes a maximum cost per arc, starting at 100 (just enough for a 30 min break). This greatly reduces the size of the graph and therefore the number of labels. Once the number of new columns is small, this factor is multiplied by 2. This is repeated until all arcs are included. While a lot of arcs are not present in the early stage, the focus on arcs that are likely to appear in good columns leads to very fast convergence to good solutions, even if there is not enough time to finish the column-generation process.

\subsubsection{Improved Elimination of Non-Dominated Labels}

Even with the previously mentioned improvements, the majority of runtime is spent figuring out which labels are non-dominated. The naive approach is to compare each new label on a node with each label in the current set of non-dominated labels. However, the runtime complexity of this method is quadratic in the number of non-dominated labels. Therefore, two different methods are explored to speed up the dominance checks. The first method is the multi-dimensional divide-and-conquer \cite{bentley1980multidimensional}: It is based on recursively solving the $k$-dimensional problem with $n$ labels by two $k$-dimensional problems of size $n/2$ and one $k-1$-dimensional problem of size $n$.

The second method is a two-stage approach based on k-d trees and bounding boxes \cite{chen2012maxima}. The two-stage approach is more effective for several reasons. First, the algorithm is based on two stages. Instead of comparing each new label with all previously non-dominated labels both ways and deleting dominated labels in the process, two passes are performed. Labels are added as long as they are not dominated, but no removal operations are executed. The second pass is performed in opposite order before expanding the labels at a node, ensuring that only non-dominated labels are expanded. Second, instead of storing labels in a list, a $k$-dimensional tree is used for storage. For each node $r$ on level $\ell$, the left child is better with respect to dimension $\ell\text{ mod }k$ and the right child is worse (or equal). Finally, each node $r$ is associated with a bound $u_r$ where each dimension contains the best value among all nodes in the subtree rooted at $r$. Therefore, if a label is not dominated by $u_r$, it is not dominated by any label in the subtree rooted at $r$, saving several comparisons.

% \clearpage
\section{Large Neighborhood Search}\label{sec:Algo}
The \textit{Large Neighborhood Search} (LNS) algorithm was introduced by Paul Shaw in 1998~\cite{Shaw1998}. The algorithm starts from an already feasible solution. The main idea is to destroy part of a solution in order to obtain a sub-problem that is easy to solve optimally or at least close to optimality. Selecting the part to destroy is done by a set $\Omega$ of \emph{destroy operators} (or \textit{destroyers}), the operator to apply is chosen randomly proportional to a given weight vector $\bm \rho$. Solving the sub-problem is done by a \emph{repair operator}, often an exact method. We accept the new solution $S'$ if $z(S')< z\left(S_\textup{bsf}\right)$, where $z$ represents the objective function value~\eqref{eq:objective} and $s_\textup{bsf}$ is the best-so-far solution. Algorithm~\ref{alg:alns} shows the pseudo-code of the algorithm.

\begin{algorithm}[t]
    \caption{Adaptive Large Neighborhood Search}
    \label{alg:alns}
    \KwIn{$\size_0$ (initial destruction size)}
    \KwOut{$S_\textup{bsf}$ (best solution found)}
    
    $\size \gets \size_0$\;
    Construct the initial solution $S_0$ using the Greedy algorithm\;
    $S_\textup{bsf} \gets S_0$\;
    Initialise the weights $\bm \rho$\;
    \While{$\text{time} < t_\textup{max}$}{
        Select destroy operator $\omega \in \Omega$ using $\bm \rho$\;
        $S' \gets r(\omega(S_\textup{bsf}, \size))$\;
        
        \If{$z(S') < z(s_\textup{bsf})$}{
            $S_\textup{bsf} \gets S'$\;
        }
        
        Update weights $\bm \rho$ and sub-problem size $k$\;
    }
    
    \Return{$S_\textup{bsf}$}\;
\end{algorithm}

\subsection{Destroy Operators}

Since our repair mechanism can only produce complete shifts, the aim of the destroyers is to select a subset of employees $E'\subseteq E$ that is removed from the current solution. The size of the sub-problem $\size=\abs{E'}$ is given to the destroy operator. We propose three distinct ways to select $E'$:

\begin{description}
	\item Employees uniform (\eu): Select $\size$ employees uniformly. 
	\item Employees weighted (\ew): $\left\lfloor\frac{\size}{2}\right\rfloor$ of the employees are selected using their cost as weight, the others uniformly. This is motivated by the fact that employees with high cost have a higher potential to benefit from reoptimization. The split is done since a combination of high-cost and low-cost shifts can have a better potential to balance the shifts in the sub-problem, e.g., by transferring some legs from the high-cost shift to an underutilised shift.
 \item Tour remover (\tr): A tour is uniformly selected and all employees that share at least one leg of this tour are removed. This process is iterated until at least $\size$ employees are removed. This operator is based on the idea of selecting employees that have something in common and therefore have a higher potential that useful recombinations of their shifts are possible, e.g., optimizing when and where a bus is handed over from one driver to the other. Note that this operator might select more than $\size$ employees because it  removes all the employees who share a tour. However,  tours are usually not shared by too many employees since this incurs extra cost, so $\abs{E'}$ does typically not exceed $\size$ by much.
\end{description}

\subsection{Repair Operators}

Once a set of removed employees $E'$ is selected, the repair mechanism needs to solve the sub-instance that is created by using all legs $\ell$ assigned to any employee $e\in E'$ together with the common data for the whole instance. % that includes the distance matrix $d$ and position specific data ($\mathit{startWork}_p$ and $\mathit{endWork}_p$). 
This sub-instance represents a complete instance of BDSP and can therefore be solved with any solution method of choice.

Since the \acrfull{bp} approach presented in Section~\ref{sec:bp} is the most powerful for small instances (it can provide an optimal solution for instances with $10$ tours within seconds), it is the best fit for solving these sub-instances.

However, as the evaluation in Section~\ref{sec:results} shows, for small instances, the results are very close to the optimum when only solving \acrfull{cg} on the root node and then solving the master problem with integrality constraint on the set of columns obtained during \gls{cg}. These solutions are often much faster, but achieve a gap of around 1\% while the following branching process only closes this remaining gap very slowly.

Therefore, we propose to drop the aim of optimally solving the sub-instance with \gls{bp}, and instead only use \gls{cg} on the root node to get very good solutions to the sub-instance very fast. In the evaluation, we compare the repair operators using Column Generation (\rcg) and full Branch and Price (\rbp).

Once the repair mechanism returns a solution consisting of employees $E^*$ that contain all bus legs from $E'$, the new solution for the full problem is provided by $(E\backslash E')\cup E^*$.

\subsection{Sub-Problem Size}

An important parameter for large Neighborhood search is the size of the sub-problem. However, the appropriate size depends on the destroy and repair operators. In the case of our system, the destroy operators are easy and fast to apply, but the complexity of Branch and Price increases rapidly with the size. Even when just applying Column Generation, the size of the RCSPP in the sub-problem still leads to considerable increases in runtime.

Therefore, based on preliminary experiments, the smallest sub-problem size in use is $\size=5$. This size can still be solved in a few seconds, so it is fast enough, but it also leads to a high number of improvements, so it is large enough to allow meaningful changes of the solution. In the process of the search this size can be increased if too many iterations without improvement occur. This indicates that a larger size might be needed to escape local optima.

We use a maximum size of $\size_{\textup{max}} = 20$ since runtime grows rapidly and for larger size too much time would be spent on each individual sub-problem. When running the algorithm, the size starts with an initial value of $\size_0$, and is increased by 1 until reaching $\size_{\textup{max}} = 20$ whenever the previous improvement was more than $\maxiter$ iterations ago. As soon as an improvement is found, $\size$ is set back to the initial value $\size_0$.
%, as suggested by \citeauthor{Blum2016} (\citeyear{Blum2016}).

\subsection{Adaptivity}
\textit{Adaptive Large Neighborhood Search} (ALNS) is an extension of LNS, where the weights $\bm \rho$  for selecting the operators are adapted dynamically based on their performance \cite{Ropke2006}.

Our method takes into account the score and the time required by  destroy operator $i$. 
At first, every component of the weight vector $\bm \rho$ is set to $\frac{1}{\abs{\Omega}}$.
The destroy operator is selected in a random way with weights $\bm \rho$ using the \textit{roulette wheel principle}:% the probability of choosing the $i$-th destroy operator is defined as follows:
\begin{equation}
    \mathbb{P}\left(\text{$i$-th operator is selected}\right) = \frac{\rho_i}{\sum_{j=1}^{\lvert{\Omega}\rvert}\rho_j}
\end{equation}
    The selected destroy operator is then applied to the current solution $S$, which results in a sub-problem that is passed to the repair operator $r$. We update the weights considering the number of successes and the total time of its selections. A similar approach was used in a related crew scheduling domain~\cite{carmo2019adaptive}.
At iteration $n$, we update the weight $\rho^n_i$ of the $i$-th destroy operator  using the following equation:
\begin{equation}\label{eq:weights_update}
	\rho^{n+1}_i=  \lambda\rho^n_i + (1-\lambda)\,\frac{\sum_{j=0}^{n}\sigma^j_i}{\sum_{j=0}^n\tau^j_i}
\end{equation}
where
$\sigma^j_i = 1$ if the $i$-th operator has improved the best-known-solution at iteration $j$, else $0$. The denominator is a sum of runtimes, so $\tau^j_i$ represents the time the $i$-th operator took for the entire process (destroying + repairing) at iteration $j$. If operator $i$ was not selected at iteration $j$, then $\sigma^j_i=\tau^j_i=0$. The real parameter $\lambda\in [0,1]$ controls the sensitivity of the weights. A value of $\lambda$ close to $0$ implies that the operator performance during the search has a large influence while a value of $1$ keeps the initial weights static.

As long as the denominator is $0$, the value of the fraction is set to $0$. In this case, $\rho^{n+1}_i= \lambda\rho^n_i$.

\section{Integration of Column Generation and Large Neighborhood Search}
\label{sec:integration}

When applying \acrfull{lns} on an optimization problem, per default the repair operators are used in a black-box fashion: The current sub-problem is fed into the operator, the corresponding solution is used to update the solution to the overall problem, and it does not matter how it was obtained. Each sub-problem is solved from scratch, and all additional information gained while solving the sub-problem is discarded every time.

However, this might actually be inefficient, as information from solving each sub-problem might be useful for future sub-problems, or it might be used beyond individual sub-problems to globally enhance the best solution. In this section, we present two novel integration techniques for combining \gls{lns} with a \acrfull{cg} repair operator that a generally applicable for this combination of solution methods.

\subsection{Column Storage}\label{sec:column_storage}

The first integration is dedicated to the reuse of information between sub-problems. Recall that each shift (column) $s$ generated by \gls{cg} contains a subset of the bus legs $L_s\subseteq L$. Each time a subproblem $i$ consisting of legs $L_i\subseteq L$ is solved, a large set of columns $S_i$ is generated, and a solution to the sub-problem $S_i^*\subseteq S_i$ is returned.

By default, the columns are regenerated for each sub-problem, however, the same columns might be generated repeatedly by multiple sub-problems. Therefore, for a potential improvement, a column storage $\hat{S}$ is introduced, and after solving a sub-problem, the update in Equation~\eqref{eq:storage_update} is performed.
\begin{equation}\label{eq:storage_update}
    \hat{S}\gets \hat{S}\cup\texttt{select}(S_i)
\end{equation}

\noindent
A subset of the newly generated columns is added to the column storage $\hat{S}$, where the selection criteria can be chosen freely, including no storage, or storing all columns.

Now, each time a sub-problem $i$ is solved, the set of columns $S_i$ can be initialized as seen in Equation~\eqref{eq:col_init}.

\begin{equation}\label{eq:col_init}
    S_i\gets\{s\in\hat{S}~|~\forall\ell\in L_s:\ \ell\in L_i\}
\end{equation}

\noindent
Each column from the storage that only contains legs that are part of the sub-problem are added to the initial set of columns, therefore, these columns do not need to be rediscovered again in the current sub-problem. Using this column reuse strategy is denoted by adding $\textsc{+r}$ to the \gls{lns} version.

\subsection{Global Background Solver}\label{sec:background}

The second improvement adds the global view of \gls{bp} into the local view based on sub-problems in \gls{lns}. The whole set of columns $\hat{S}$ can be used in a global master problem as described in Section~\ref{sec:master}. Since in \gls{lns} we do not aim to solve the problem exactly, there is no need to solve the relaxed master problem, instead the integer problem can be solved to get the best possible solution for the current set of stored columns.

However, solving increasingly larger integer master problems takes time and memory. Therefore, we propose to use a second thread for this improvement, using the first thread entirely for \gls{lns}, while the background thread repeatedly solves the master problem for $\hat{S}$. At the end of every repairing phase of \gls{lns}, the algorithms checks whether the solution from the second thread is better than the current best: 
\[s_\text{bsf}\gets \argmin\left(z(s'), z(s_\text{bsf}), z(s_\text{bg})\right)\] where $s_\text{bg}$ is the solution from the second thread, $s_\text{bsf}$ is the best-so-far and $s'$ is the solution after the repairing phase, as described in Algorithm~\ref{alg:alns} . If $s_\text{bg}$ is better, it replaces $s_\text{bsf}$, and \gls{lns} continues from this improved solution. This is only a mild form of parallelization that is easily applicable with current multi-core machines, but can be very beneficial to further improve the joint performance of the methods.

As the background solver repeatedly solves similar problems, and new columns are frequently added, two more considerations are relevant. First, columns for this solver are never removed, but only added, allowing a warm start from the previous result in each solving cycle. Second, the master problem might not be solved to optimality, but stopped according to a given criterion, since incorporating new columns might be more beneficial than spending more time on the current cycle. We propose to use a timeout $t_{bg}$, but at least solve the root node of the MIP, before ending the current cycle. Adding the background solver is denoted as $\textsc{+b}$ for the LNS version.

This improvement using the background solver can also be used for \acrlong{bp} on its own. By default, the integer master problem is solved there whenever \acrlong{cg} on a node is finished (line 26 in Algorithm~\ref{alg:bp}), or in case even the first node runs into timeout (line 19). However, as each integer solution is a global upper bound independent from the current position in the branching tree, this process can be done in the background solver as well on the set of columns $S$. We call this version \bpbg in the evaluation.

\subsection{Selection of Columns to Store}
\label{sec:col_select}

A critical parameter for the proposed improvements is the selection of columns to store via the \texttt{select} function. The easy way is to store all new columns, we refer to this option as \textsc{(f)} (full set). However, this might use a considerable amount of memory. Further, the time to search for useful columns for the current sub-problem might counter the benefit of not having to rediscover them, and the background master problem might get excessively large.

A very lightweight alternative would be to only select the best subset $S_i^*$ for each sub-problem. We denote this version as \textsc{(b)} (best). The focus on only the best columns will keep the size of $\hat{S}$ low, but ideally still preserve very good solutions, while the selection of different sub-problems over time should still provide some diversity.

\afterpage{\clearpage}
\section{Evaluation}
\label{sec:results}

All executions were performed on a cluster with $11$ nodes using Ubuntu 22.04.2 LTS (GNU/Linux 6.8.0-48-generic x86\_64) set up for maximum reproducibility \cite{fichte2024parallel}.
Each node has two  Intel
Xeon E5-2650 v4 processors (max \qty{2.20}{\giga\hertz}, 12 physical cores, no hyperthreading). Unless otherwise specified, we use a single thread and up to  \qty{25.6}{\giga\byte} of RAM. Unless stated otherwise, 
all experiments are performed with a fixed budget of $\qty{1}{\hour}$ of CPU-time. Deterministic algorithms are executed once, non-deterministic ones are executed $10$ times for each instance. Random seeds for different runs are recorded for reproducibility.

The implementation is written in Python, executed with PyPy 7.3.17 for speed using Python 3.10.14. \acrlong{bp} was implemented in Java, using OpenJDK 23.0.1, and Gurobi 12.00 for the master problem.  The figures were generated using Matplotlib~\cite{Hunter:2007} v3.10.0 and Seaborn~\cite{Waskom2021} v0.13.2.

\paragraph{Software Changes.}

To get a fair comparison, we reran the experiments from the previous publications \cite{kletzander2021branch,mazzoli2024investigating}. We updated previously used software versions, but the largest change was switching from CPLEX to Gurobi. This was due to the fact that at least the Java interface of CPLEX is very memory-hungry, leading to frequent out-of-memory errors for larger instances, and frequent crashes of the CPLEX-Java interface when Java was in the process of garbage collection (freeing unused memory). In contrast, Gurobi works on all instances, even very large ones, without running out of memory.

\paragraph{Evaluation Metric.}
To have a metric quality that does not scale with the dimension of an instance, we evaluate the quality of a solution using the relative gap ($\GAP$) compared to the best-known solution:
\begin{equation}
\GAP(x) = \frac{z(x)-z\bigl(x_\textup{bks}\bigr)}{z\bigl(x_\textup{bks}\bigr)}\cdot 100,
\end{equation}
where $x$ is the solution and $x_\textup{bks}$ is the best-known solution among all methods evaluated in this article.

\paragraph{Instances and Initial Solutions.}
We use the publicly available sets of benchmark instances provided by previous work \cite{kletzander2020solving,Kletzander2022}\footnote{\url{https://cdlab-artis.dbai.tuwien.ac.at/papers/sa-bds/}}. 
There are $65$ instances in $13$ sizes, ranging from around $10$ to around $250$ tours. Note that the last 3 sizes have much larger size changes than the previous sizes. 

The instances in this paper use $\mathit{startWork}_p=15$ and $\mathit{endWork}_p=10$ at the depot ($p=0$).
These values are $0$ for other positions. For the given instances, the number of legs is proportional to the number of bus tours with approximately $n_\textup{legs} \approx 10\cdot n_\textup{tours}$. 

The initial solutions required by \gls{lns} have been generated using a greedy construction method~\cite{kletzander2020solving}, assigning bus legs to the employee where the lowest additional cost is incurred, or to a new employee if this would incur an extra cost of at most $500$ compared to the best existing employee assignment.

\gls{bp} starts from a very simple solution where each bus leg forms its own shift as starting from the greedy solution did not improve the results.

\subsection{Branch and Price}

This subsection provides the evaluation of important \gls{bp} components and their effect, as well as the comparison of using a single thread (\bpplain), parallelization of the MIP (\bpmulti), and mild parallelization using the background thread (\bpbg).

\subsubsection{Effects of Subproblem Improvements}

While all the design choices and parameter settings have been carefully tested, this section focuses on the effects of our crucial subproblem improvements. Note that these experiments were done on a faster machine with an i7-6700K with 4x4.0 GHz and with CPLEX as the MIP solver.

We highlight the benefits for splitting the subproblem on the example of instance $40\_16$. At the root node the subproblem is solved 197 times. Only 76 times the uncentered 30-min break graph is solved and only 14 times the centered 30-min break graph. However, the runtimes for the different graphs are 4 ms, 1.5 s and 1.9 s respectively at the beginning (arc throttling) and 16 ms, 7.7 s and 51.1 s at the end (all arcs). Even though easier subproblems only provide parts of the new shifts, their runtime advantage makes it worth only going to the more complex subproblems when necessary. 

\begin{figure}[!th]
   \centering
\includegraphics[width=\scalingFactor\textwidth]{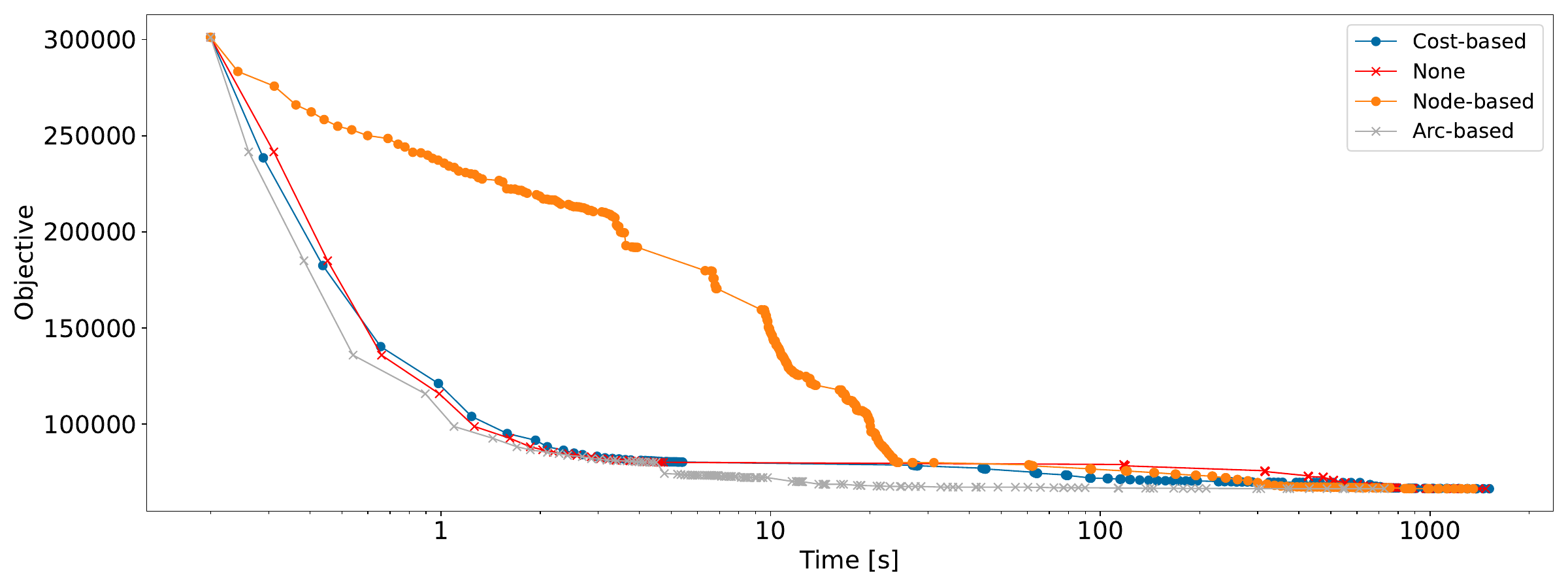} 
\caption{Throttling approaches for $40\_16$ }
\label{fig:throttle}
\end{figure}

Figure~\ref{fig:throttle} shows the comparison of cost-based throttling (100 best regarding cost per node, factor 10 increase), no throttling, node-based throttling (100 first per node, factor 10 increase), and arc-based throttling (arc cost limited to 100, factor 2 increase) on the objective of the reduced master problem over time (logarithmic time axis to highlight the differences early in the search). The results show a significant overhead of cost-based throttling (taking the longest time overall), and weak improvements early but more efficiency later for node-based throttling. Arc-based throttling, however, clearly shows superior performance starting to drop significantly in objective value compared to other approaches in less than 10 seconds. Its overall time is 727 seconds compared to 1354 seconds for node-based throttling, the second best approach.

Regarding the dominance algorithm, one run of the largest subproblem on instance $40\_16$ takes 290 seconds with the default dominance algorithm, 293 seconds with the multidimensional divide-and-conquer algorithm, and 51 seconds with the two-stage algorithm. Overall, for default dominance the root node terminates in only 25 compared to 33 instances for the two-stage algorithm within the time limit.

\subsubsection{Single- and Multi-Threaded Performance}

By default \bpplain uses the following schedule in a single thread: It allows up to one hour of \gls{bp}, but if \acrlong{cg} at the root node is not finished by then, it allows up to one more hour to solve the integer master problem to provide at least a best effort solution, which is very good in many cases. Still, this approach requires two separate timeouts, and it is not clear in the beginning if the full time will be needed on any of them. Recall that an integer solution is calculated at each node in the branching tree, therefore, the second phase is not needed if the root node is completed within the first hour.

We compare \bpplain to the following multi-threaded versions in this section:
\begin{description}
    \item \bpmulti\quad Multi-threaded MIP and LP solving: Since a considerable amount of time is spent solving the relaxed and integer master problem, we use $8$ threads for Gurobi to speed it up.
    \item \bpbg\quad MIP solving in the background: The main thread deals with the RMP and solving the sub-problem, while solving the integer problems is shifted to a background thread as described in Section~\ref{sec:background}. This has the advantage that no two separate phases are needed if the root node is not completed.
\end{description}

\begin{table}[!th]
\small
\centering
\caption{\gls{bp} results grouped by size}
\label{tab:results_bp_sizes}
\begin{tabular}{l*{3}{S[table-format=6.1]cS[table-format=4.1]}}
\toprule
 &\multicolumn{3}{c}{\bpplain}&\multicolumn{3}{c}{\bpmulti}&\multicolumn{3}{c}{\bpbg} \\
\cmidrule(lr){2-4} \cmidrule(lr){5-7} \cmidrule(lr){8-10}
Size &{Best}&{Opt. Gap}&{Time [s]}&{Best}&{Opt. Gap}&{Time [s]}&{Best}&{Opt. Gap}&{Time [s]}\\
\midrule
10 & \bfseries 14709.2 & 0.0 & 8.9 & \bfseries 14709.2  & 0.0 & 8.4 & \bfseries 14709.2  & 0.0 & 7.9 \\
20 &  30299.4 & 0.0 & 1430.1 & 30299.4 & 0.0 & 1320.7 & \bfseries 30294.6 & 0.0 & 1410.4 \\
30 & \bfseries 49846.4 & 0.4 & 3605.3 & 49848.0 & 0.4 & 3604.5 & \bfseries 49846.4 & 0.4 & 3605.8 \\
40 & 67017.0 & 0.4 & 3612.4 &\bfseries  67009.2 & 0.3 & 3600.2 & 67016.4 & 0.3 & 3604.1 \\
50 & 84338.8 & 0.4 & 3627.7 & 84344.8 & 0.4 & 3603.1 & \bfseries 84332.4 & 0.4 & 3626.8 \\
60 & \bfseries 99754.6 & 0.4 & 4058.5 &\bfseries  99754.6 & 0.4 & 3633.6 & 99720.4 & 0.4 & 3602.3 \\
70 & 118337.6 & - & 5959.4 & \bfseries  118285.6 & - & 5147.8 & 118472.8 & - & 3601.5 \\
80 & 134925.8 & - & 6244.4 &\bfseries  134661.0 & - & 5243.1 & 134978.0 & - & 3601.5 \\
90 & 150292.8 & - & 7200.9 & \bfseries 150172.0 & - & 6833.1 & 150804.6 & - & 3601.7 \\
100 & 168554.2 & - & 6624.6 &\bfseries  168024.4 & - & 6301.4 & 170070.0 & - & 3601.4 \\
150 & 273629.2 & - & 7200.4 & \bfseries 266465.2 & - & 7200.5 & 280936.4 & - & 3601.6 \\
200 & 380518.0 & - & 7202.1 & \bfseries 371942.0 & - & 7200.6 & 390941.4 & - & 3603.7 \\
250 & 520063.4 & - & 7217.8 & \bfseries 498122.4 & - & 7208.5 & 534937.2 & - & 3603.1 \\
\bottomrule
\end{tabular}
\end{table}

\begin{figure}[!th]
    \centering
    \includegraphics[width=\scalingFactor\textwidth]{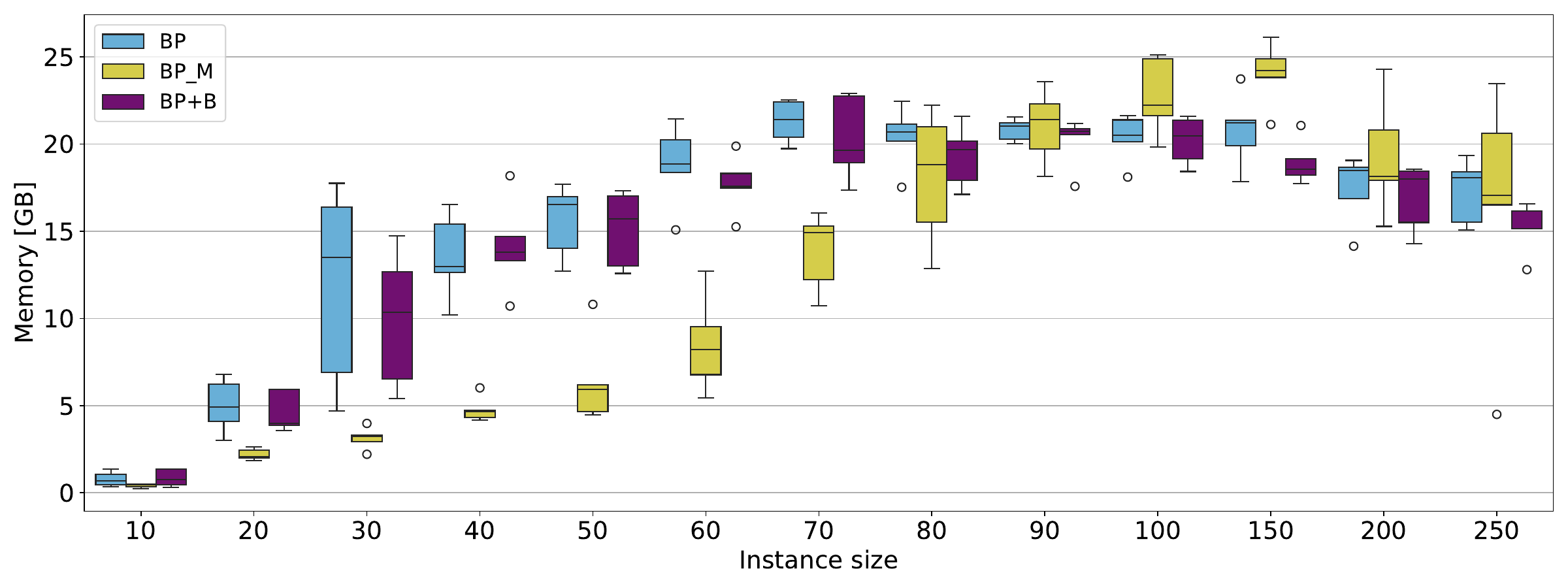}
    \caption{Memory usage of different B\&P variants}
    \label{fig:ram_bp}
\end{figure}

We compare the three variants of \gls{bp} and display the results in Table \ref{tab:results_bp_sizes}. Each entry shows the average of the results for 5 instances of the same size. Column \emph{OG} shows the optimality gap in percent. There was only one run per method since \gls{bp} is deterministic (except for small differences due to thread timing with \bpbg). Figure~\ref{fig:ram_bp} shows the maximum memory use. We note that in most cases, especially regarding the smaller instances where \gls{bp} is most suited, there is no significant difference in the results, however, there are several relevant observations. 

Starting from size 70, \bpmulti consistently provides the best results among the options, with especially large gaps starting at size 150. While there is no indication that using the extra threads helps with solving the RMP, this shows that solving the integer master problem benefits from the additional computational resources, especially for larger problems. This also shows in the runtimes, where the time for the second phase is often much shorter. While these are significant changes, they come at the cost of excessive additional resources (1 vs. 8) threads, and show most benefit on instances where \gls{bp} is outperformed by \gls{lns} anyway, making it not worth the extra resources. Of interest regarding memory use is that this variant, compared to the others, uses way less memory for instances of size 30 to 70, but overtakes the other variants for larger instances.

The mild parallelization in \bpbg again shows very similar results for the smaller instances, while the results for very large instances are slightly worse. Since the integer master problem in the background does not immediately include the most recent new columns, the final result is computed from a slightly inferior set of columns compared to the other methods. However, the great advantage is that the extra runtime of the second phase is not added on top, but in parallel. Therefore, for larger instances, the runtime is half of \bpplain, while the computational effort is the same. Again, however, since the focus of \gls{bp} is on smaller instances, this effect is not so relevant in practice, and we use \bpplain for the comparisons with other methods.

A big advantage of \gls{bp} is that for smaller instances (where the root node is completed), a bound on the solution quality can be provided. Results for all instances of size 10, and 4 out of 5 instances of size 20 are proven optimal, and very low bounds of up to $0.4\,\%$ are provided for instances of up to size 60 (for all instances up to size 50 and 4 out of 5 of size 60). While Lagrangean bounds can also be computed for larger instances, they are too weak and not reported here.

\subsection{Large Neighborhood Search}

To select the LNS parameters, we thoroughly analyzed the impact of different algorithmic components on a subset of instances from the benchmark set. Since all algorithms show similar performance on instances of the same size, we chose one instance from each size, skipping the smallest size that can be solved to optimality with BP in seconds. Therefore, we used $12$ instances in this part of the evaluation; each result is the average of 5 runs.

For the experiments in this section, we used an earlier setup with Java (OpenJDK 14.0.1) and CPLEX 12.10 to solve the master problem. The choice of solver does not impact the experimental results, as the behavior of the MIP solver only differs for larger problems, while the sub-problems in \gls{lns} are fairly small.

We investigate the following components:
\begin{enumerate}
    \item The repair mechanism: \rbp or \rcg
    \item The initial destruction size $\size_0$
    \item Max. number of iterations without improvement $\maxiter$
    \item The destroyer selection
    \item The role of adaptivity
\end{enumerate}

\subsubsection{Repair Operator Selection}

\begin{figure}[p]
    \centering
    \begin{subcaptionblock}{\textwidth}
        \includegraphics[width=\scalingFactor\textwidth]{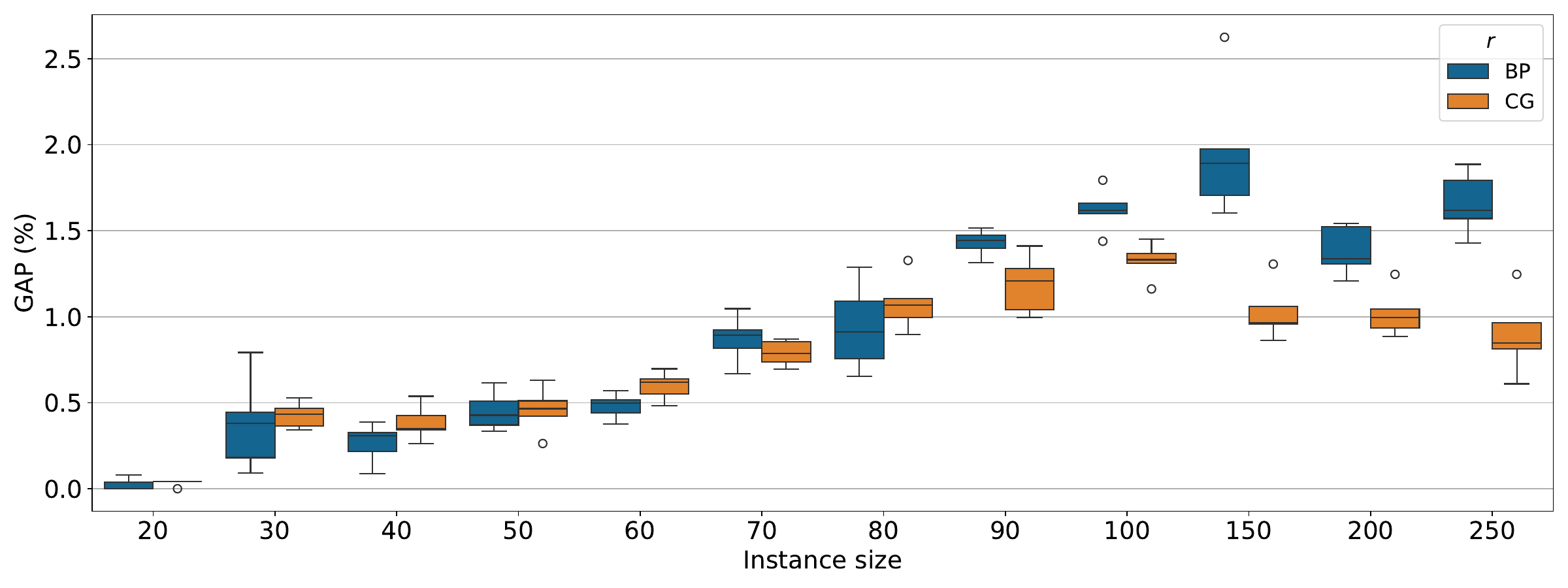}
        \subcaption{GAP for different repair operators}
        \label{fig:BPvsCG}
    \end{subcaptionblock}\\
    \begin{subcaptionblock}{\textwidth}
        \includegraphics[width=\scalingFactor\textwidth]{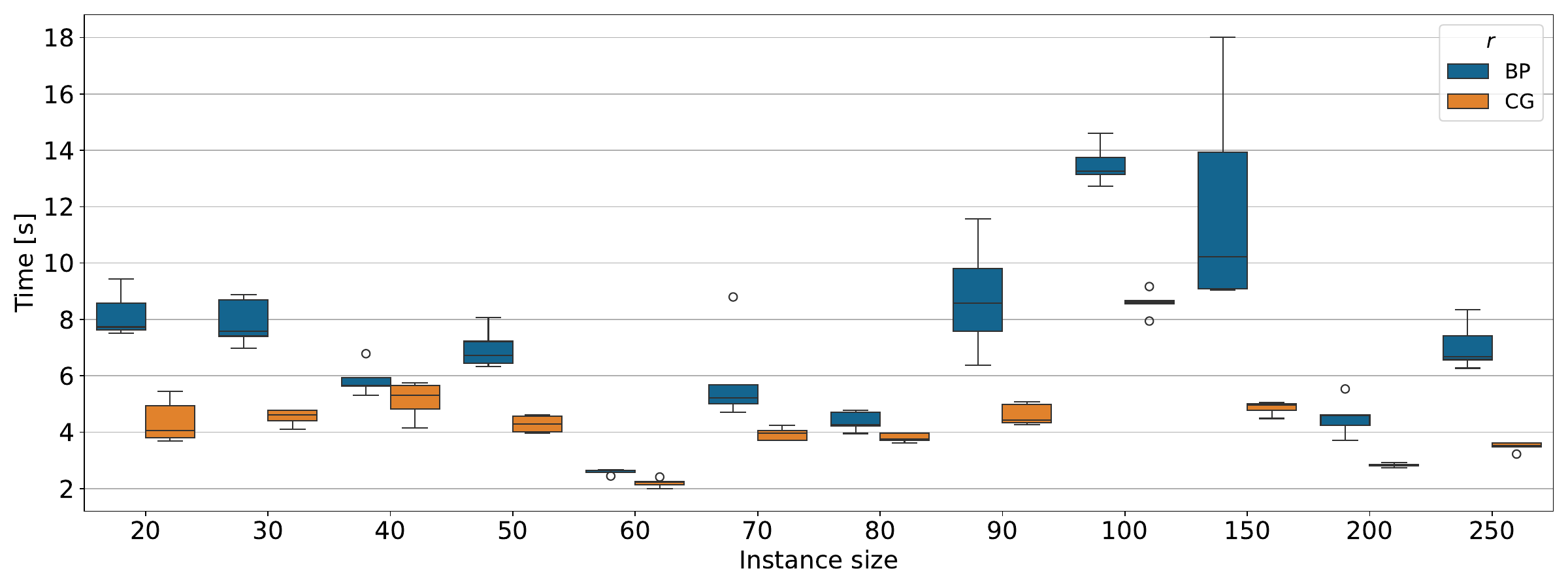}
        \subcaption{Average CPU-time of different repair operators}
        \label{fig:BPvsCGtime_avg}
    \end{subcaptionblock}\\
    \begin{subcaptionblock}{\textwidth}
        \includegraphics[width=\scalingFactor\textwidth]{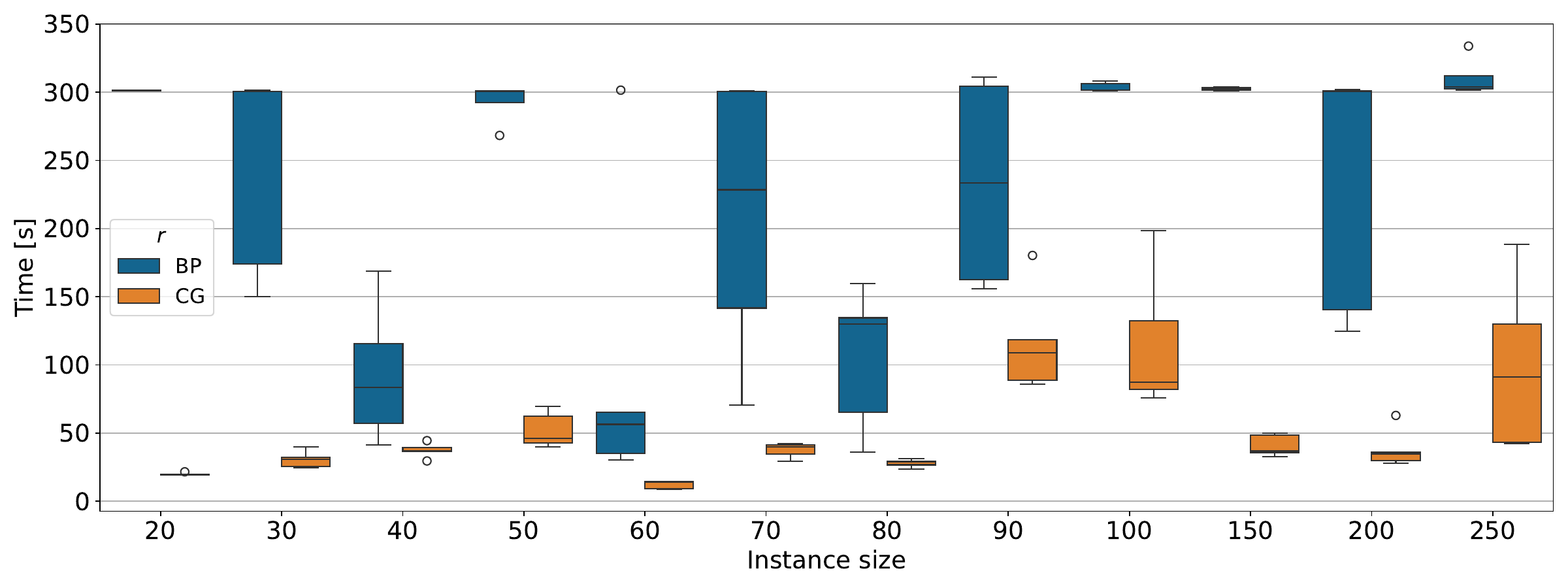}
        \subcaption{Maximum CPU-time of different repair operators}
        \label{fig:BPvsCGtime}
    \end{subcaptionblock}
    \caption{Comparison of repair operators}
    \label{fig:merged_BPvsCG}
\end{figure}

For fixed $k_0=10$, $\maxiter=50$, and equal selection of all destroyers, we compared \rbp and \rcg. Initial experiments showed that the performance of destroy and repair is rather independent from each other, which enables separate evaluation. Each repair operation has a maximum budget of $\qty{5}{\minute}$, but is expected to usually terminate much faster.

Results are shown in Figure~\ref{fig:BPvsCG}. The performance is very similar for the smaller instances, but \rcg is clearly the better choice for larger instances, showing that the extra time used for repairing in \rbp is not justified. Figures~\ref{fig:BPvsCGtime_avg} and~\ref{fig:BPvsCGtime} show that \rbp usually takes longer than \rcg. While the average time in both cases is under $\qty{10}{\second}$ for most instances, \rcg shows significantly lower average values. \rbp often reaches the time budget of $\qty{5}{\minute}$, while \rcg is mostly below $\qty{2}{\minute}$, showing a much better worst case behavior. 

\subsubsection{Initial Destruction Size}

\begin{figure}[!th]
    \centering
    \includegraphics[width=\scalingFactor\textwidth]{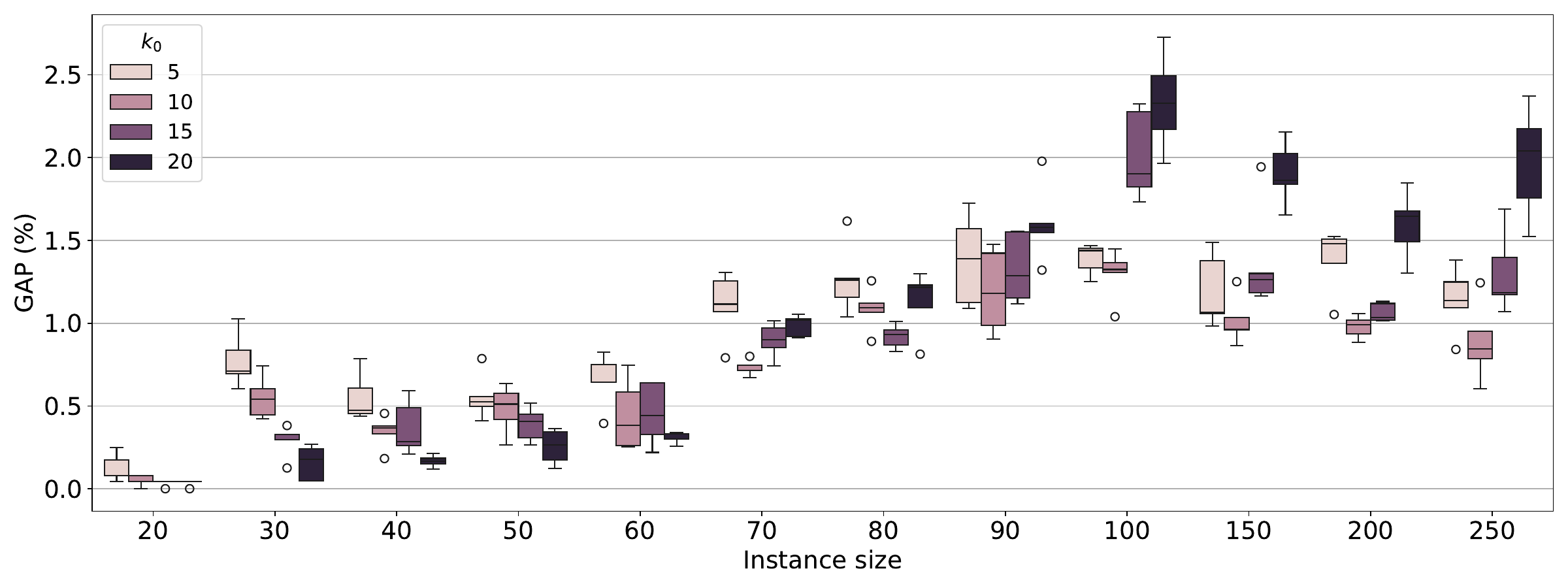}
    \caption{GAP for different values of $\size_0$}
    \label{fig:fixed_size}
\end{figure}

There are several options regarding the destroyers, first we investigated the initial destruction size $\size_0$, fixing all other parameters. The size remained constant, \rcg and all destroyers are used, and $\rho_i=1/3$ without adaptivity.

We tested sizes $\size_0\in\{5,10,15, 20\}$. 
Figure~\ref{fig:fixed_size} shows the results. While $\size_0=20$ performs slightly better for smaller instances, it is outperformed on larger instances. Overall, $\size_0=10$ seems best for the large instances which are the main focus of \gls{lns}, therefore, we fix $\size_0=10$.

\subsubsection{Number of Iterations Without Improvement}

\begin{figure}[!th]
    \centering
    \begin{subcaptionblock}{\textwidth}
        \centering
        \includegraphics[width=\scalingFactor\textwidth]{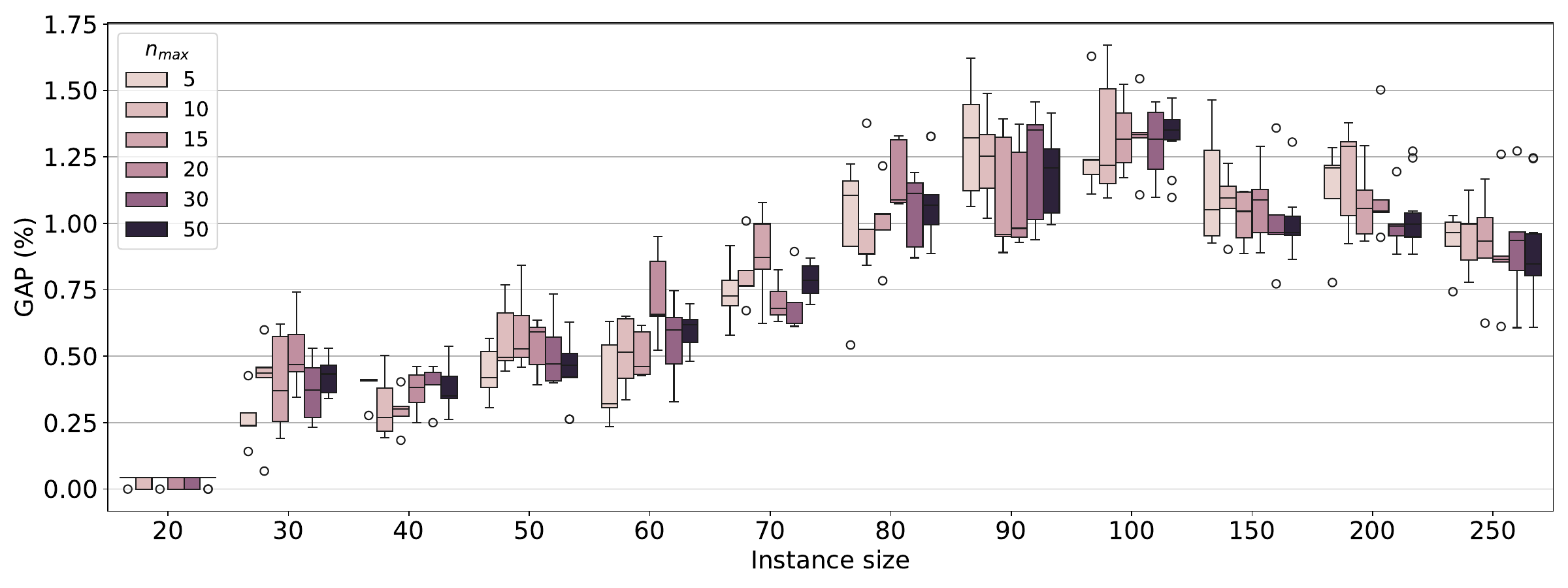}
        \subcaption{GAP for different values of $\maxiter$}
        \label{fig:dynamic_size}
    \end{subcaptionblock}\\
    \begin{subcaptionblock}{\textwidth}
        \centering
        \includegraphics[width=\scalingFactor\textwidth]{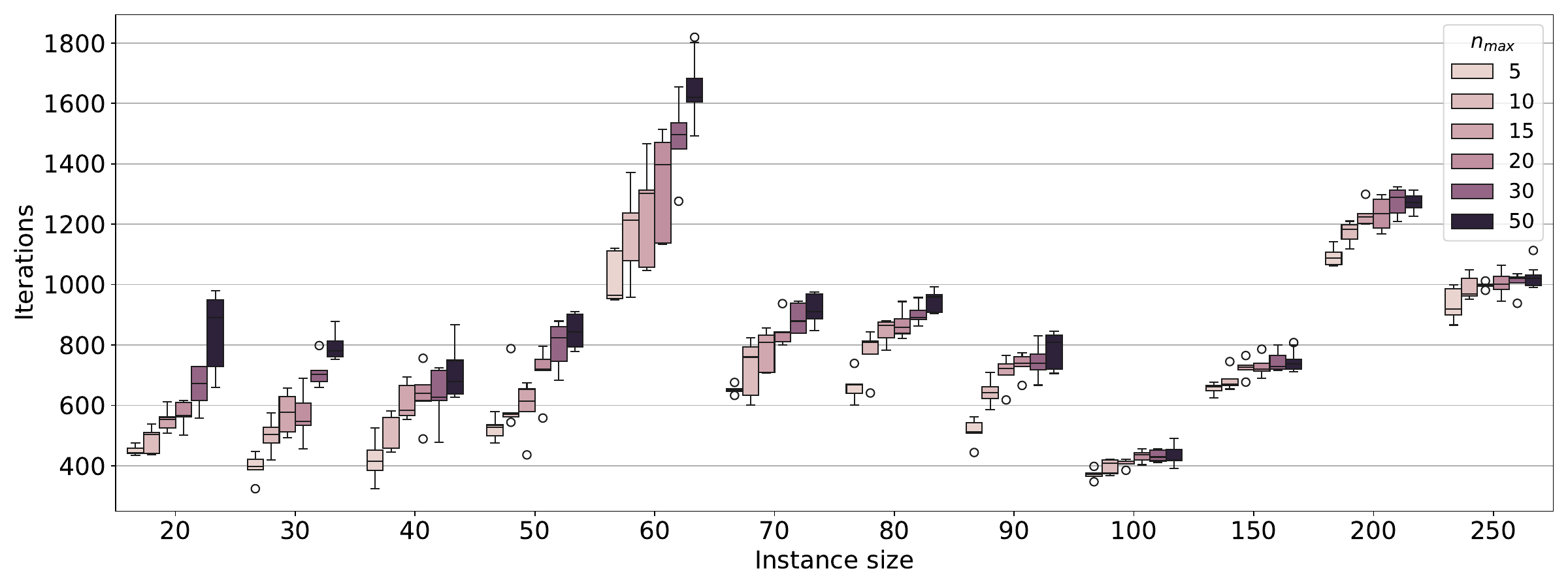}
        \subcaption{Number of iterations for different values of $\maxiter$}
        \label{fig:starting_size}
    \end{subcaptionblock}
    \caption{Comparison for different values of $\maxiter$}
    \label{fig:merged_dynamic_iterations}
\end{figure}

Next, we investigated increasing the size $\size$ every $\maxiter$ iterations without improvement by $1$, until reaching the upper bound $\size_{\textup{max}}=20$ or finding an improvement.
%As soon as we find an improvement, the value of $\size$ is set back to the initial value $\size_0=10$, as suggested by \citeauthor{Blum2016} (\citeyear{Blum2016}).

We tested $\maxiter\in\{5, 10, 15, 20, 30, 50\}$, 
but found no significant difference among them, as shown in Figure~\ref{fig:dynamic_size}. We decided to set $\maxiter=50$, since it still allows to increase the size when needed, but does not increase it very often. We tried starting with different values for $\size_0$, but found similar results, the initial size is more important than the step. 

%We noted that starting with $\size_0=0$ or $\size_0=15$ does not have a huge impact on the performance of the algorithm, even if we change the number of maximum iterations without improvement

%We also tried the initial size $\size_0=15$ and we compare the results with $\size_0=10$. However, we found that the average GAP for $\size_0=10$ was lower than $\size_0=15$. Hence, we kept it.

%Figure~\ref{fig:starting_size-boxplot} shows us that the first quartile of the distribution of the GAPs is generally lower for $\size_0=10$.

Figure~\ref{fig:starting_size} shows the impact of $\maxiter$ on the number of iterations (\rcg calls). Larger values of $\maxiter$ imply less frequent size changes, so more iterations. Of note is that for the larger instances, the size barely changes, as improvements are frequently found even with the initial size until timeout.

\subsubsection{Destroyer Selection}

\begin{figure}[!th]
    \centering
    \includegraphics[width=\scalingFactor\textwidth]{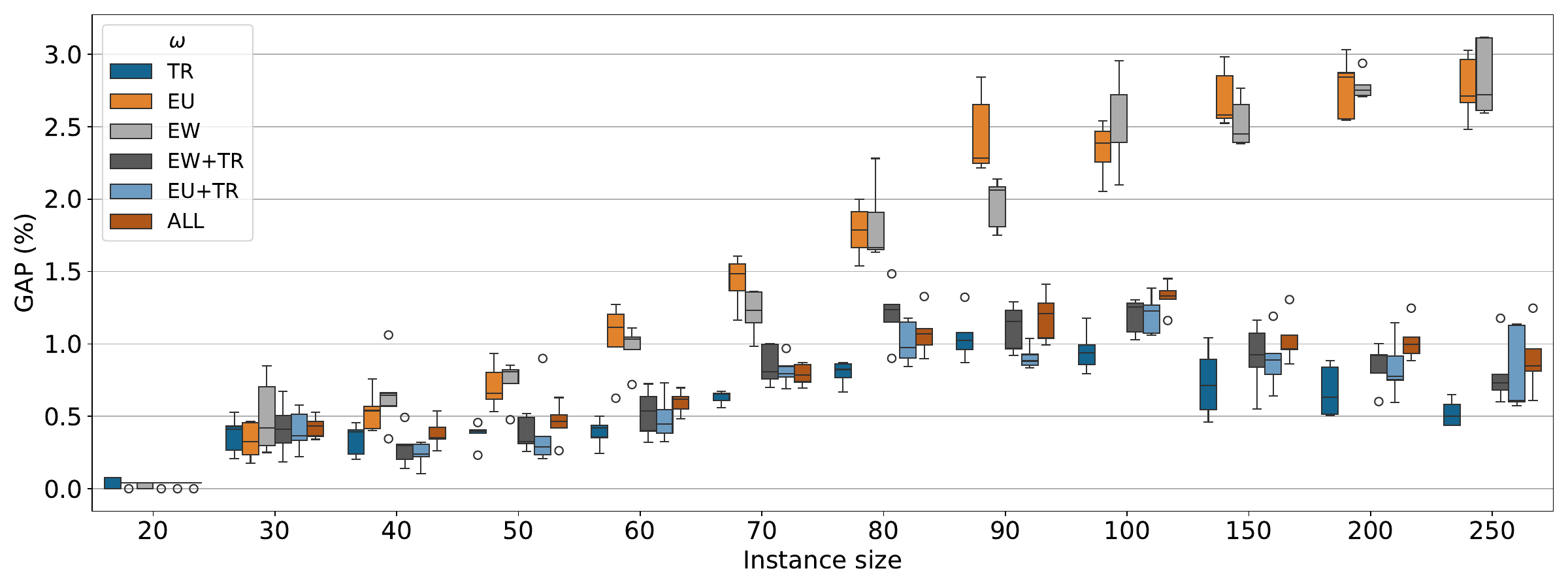}
    \caption{GAP for different subsets of destroyers}
    \label{fig:destroyers-GAP}
\end{figure}

\begin{figure}[!ht]
    \centering
    \includegraphics[width=\scalingFactor\textwidth]{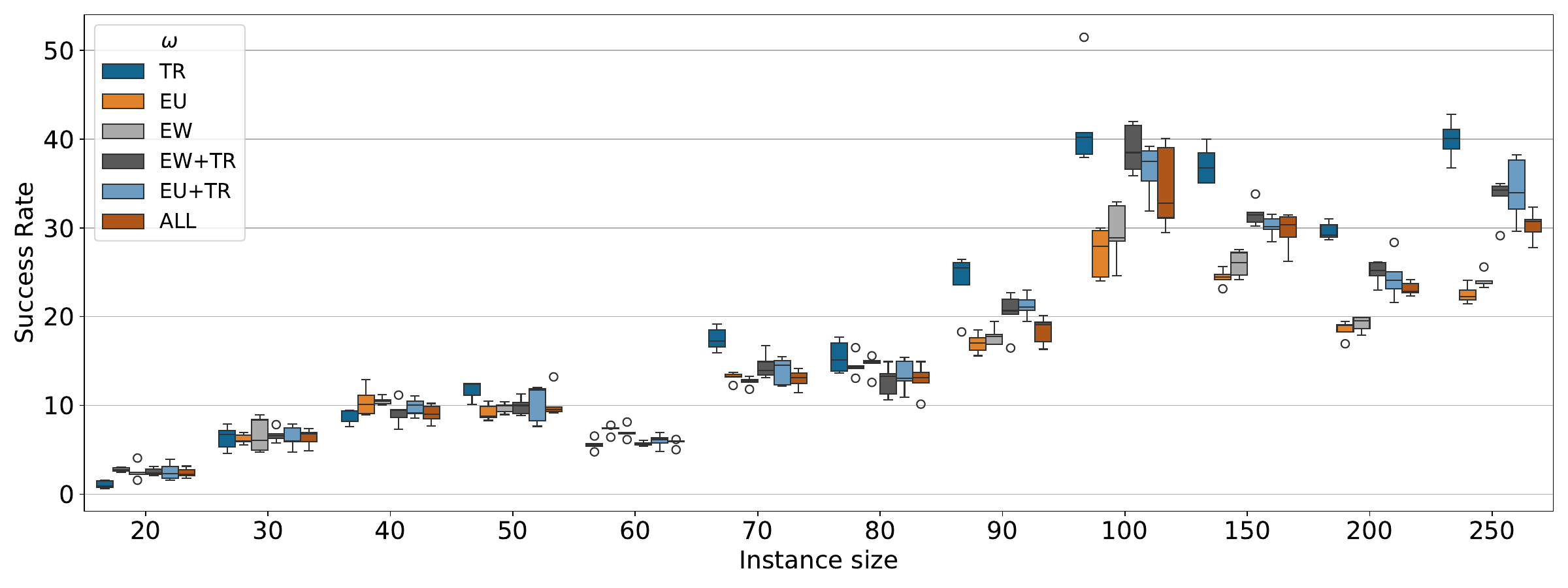}
    \caption{Success rate for different subsets of destroyers}
    \label{fig:destroyers_success}
\end{figure}

In order to understand the impact of the destroy operators, we tested all the $7$ possible combinations of them. Figure~\ref{fig:destroyers-GAP} shows that \tr has the biggest impact on the performance. Is shows the best results on its own, with very similar results using it in any other combination, while all combinations without \tr show significantly worse performance, with higher divergence among larger instances.

The advantage of selecting employees that share tours is that the sub-problems are more likely to allow meaningful optimizations. This hypothesis is further backed by the success rate (percentage of iterations where the current solution could be improved) in Figure~\ref{fig:destroyers_success}, which shows that in general for larger instances more improvements in sub-problems can be found, but especially using just \tr has a higher success rate than any other set of destroyers.

In summary, the best choice for the parameters is \tr as operator, an initial destruction size of $\size_0=10$, and $\maxiter=50$ iterations without improvements.

While only using \tr is the best choice, we further investigated several non-uniform weight distributions with high weights for \tr. Denoting the weights as $(\rho_{\omega_\textsc{eu}}, \rho_{\omega_\textsc{ew}}, \rho_{\omega_\textsc{tr}})$, we conducted experiments with $(5, 5, 90)$, $(10, 10, 80)$, and $(25, 25, 50)$. The first two where very similar to \tr only, while $(25,25,50)$ started to get slightly worse.

% In the previous section we saw that updating weights does not improve the average GAP. Therefore, we consider only the static LNS.

% Now that we fixed our parameters, we can think of our algorithm as being solely influenced by the destroyer weights.

% Henceforth, we refer to LNS with specific weights, denoted as $\texttt{LNS}(w_1, w_2, w_3)$, where the weights are assigned to the operators \texttt{eu
% }, \texttt{ew}, and \texttt{tr} respectively.
% For instance, we represents a version employing only the \tr operator as $\texttt{LNS}(0, 0, 100)$.

% The critical question now is how the \tr-only variant performs against other (static) choices of the weights.

% We conducted experiments with $\texttt{LNS}(5, 5, 90)$, $\texttt{LNS}(10, 10, 80)$, and $\texttt{LNS}(25, 25, 50)$.

% Even if the results are all very similar to each other, we can see that the results start to get worse with the choice $(25, 25, 50)$.

\subsubsection{Adaptivity}

\begin{figure}[!th]
    \centering
    \includegraphics[width=\scalingFactor\textwidth]{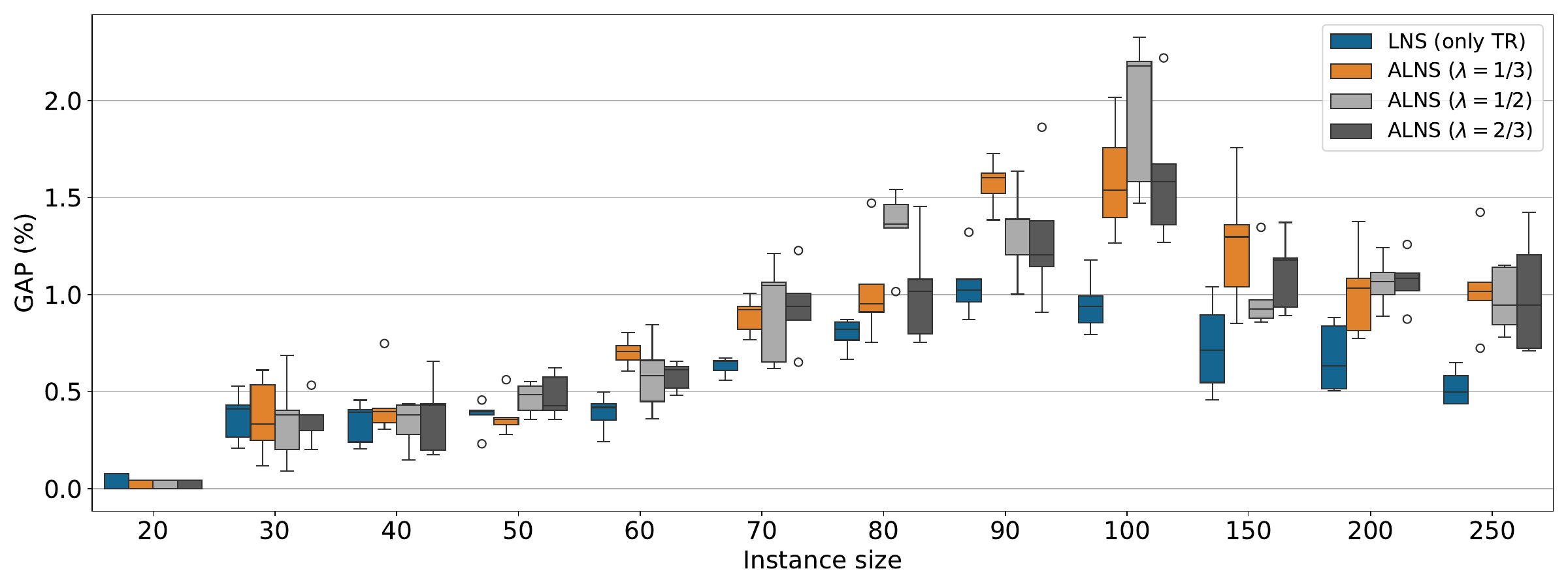}
    \caption{GAP for adaptive and static LNS}
    \label{fig:ALNSvsLNS}
\end{figure}

To investigate the impact of adaptivity, we conducted experiments by changing the parameter $\lambda$ in~\eqref{eq:weights_update}, considering all three destroy operators.
We tested three different values for $\lambda$: $\frac{1}{3}, \frac{1}{2},$ and $\frac{2}{3}$.

Figure~\ref{fig:ALNSvsLNS} suggests that the adaptivity does not improve the average GAP with respect to the solely \tr, and different values of $\lambda$ do not show significant difference.

\subsection{Integration of CG and LNS}

In this section, we evaluate and compare seven distinct variants of the Large Neighborhood Search algorithm with different levels of integration with \acrlong{cg}, summarized in Table~\ref{tab:lns_variants}. 

\begin{table}[htp]
    \caption{LNS variants}
    \label{tab:lns_variants}
    \centering
    \begin{tabular}{l*{3}{c}}
    \toprule
    Name & Column Reuse & Background Solver & \texttt{select} \\
    \midrule
     \lns & no & no & - \\
     \textsc{lns+r(b)} & yes & no & best \\
     \textsc{lns+r(f)} & yes & no & full \\
     \textsc{lns+b(b)} & no & yes & best \\
     \textsc{lns+b(f)} & no & yes & full \\
     \textsc{lns+rb(b)} & yes & yes & best \\
     \textsc{lns+rb(f)} & yes & yes & full \\
    \bottomrule
    \end{tabular}
\end{table}

The \textsc{+r} variants use the stored columns to initialize each sub-problem according to Equation~\eqref{eq:col_init}. The \textsc{+b} options apply a background MIP solver according to Section~\ref{sec:background} on the set of stored columns $\hat{S}$ with a timeout $t_{bg}=\qty{1}{\min}$ for every iteration. For each combination, two different choices for the selection of columns to store (function \texttt{select}) are evaluated according to Section~\ref{sec:col_select}. The first option \textsc{(b)} is to store only the best columns $S_i^*$ for each sub-problem, the other option \textsc{(f)} to store the full set $S_i$.

\subsubsection{Impact on LNS Metrics}

\begin{figure}[p]
    \centering
    % N OF ITERATIONS
    \begin{subcaptionblock}{\textwidth}
        \centering
        \includegraphics[width=\scalingFactor\textwidth]{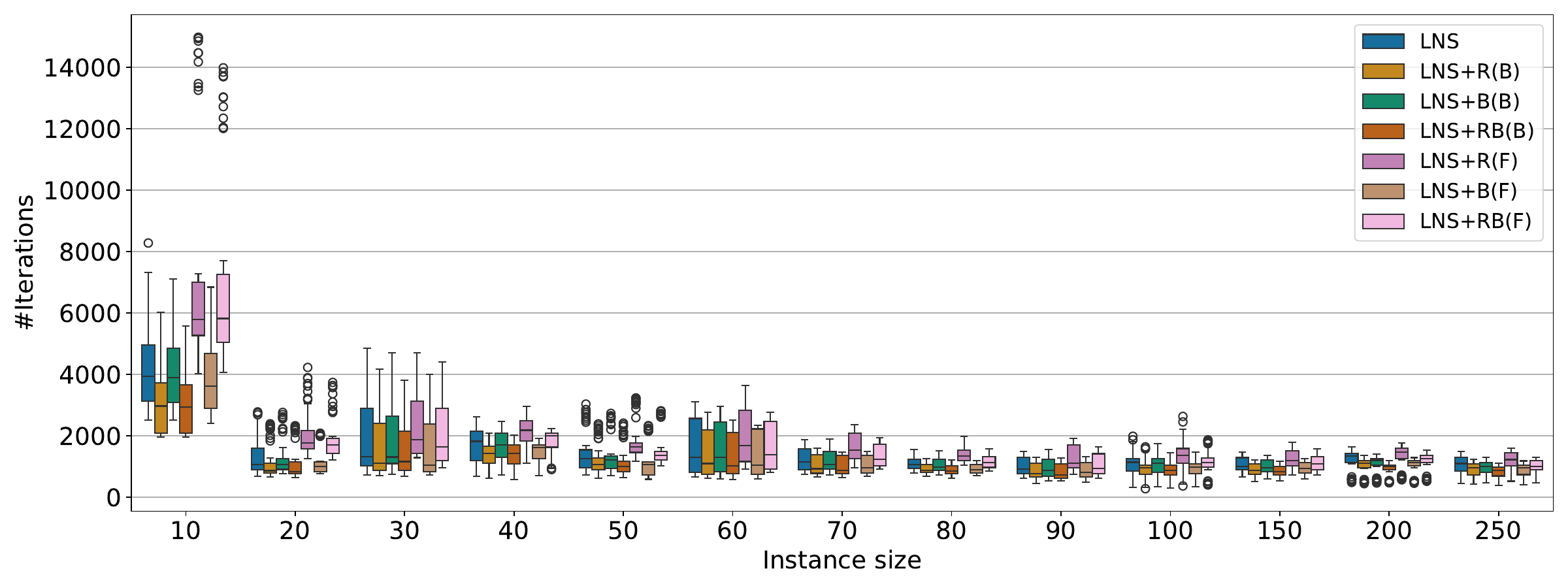}
            \caption{Number of iterations}\label{fig:boxplot_iterations}
    \end{subcaptionblock}\\
    % REPAIRING SPEED
    \begin{subcaptionblock}{\textwidth}
        \centering
        \includegraphics[width=\scalingFactor\textwidth]{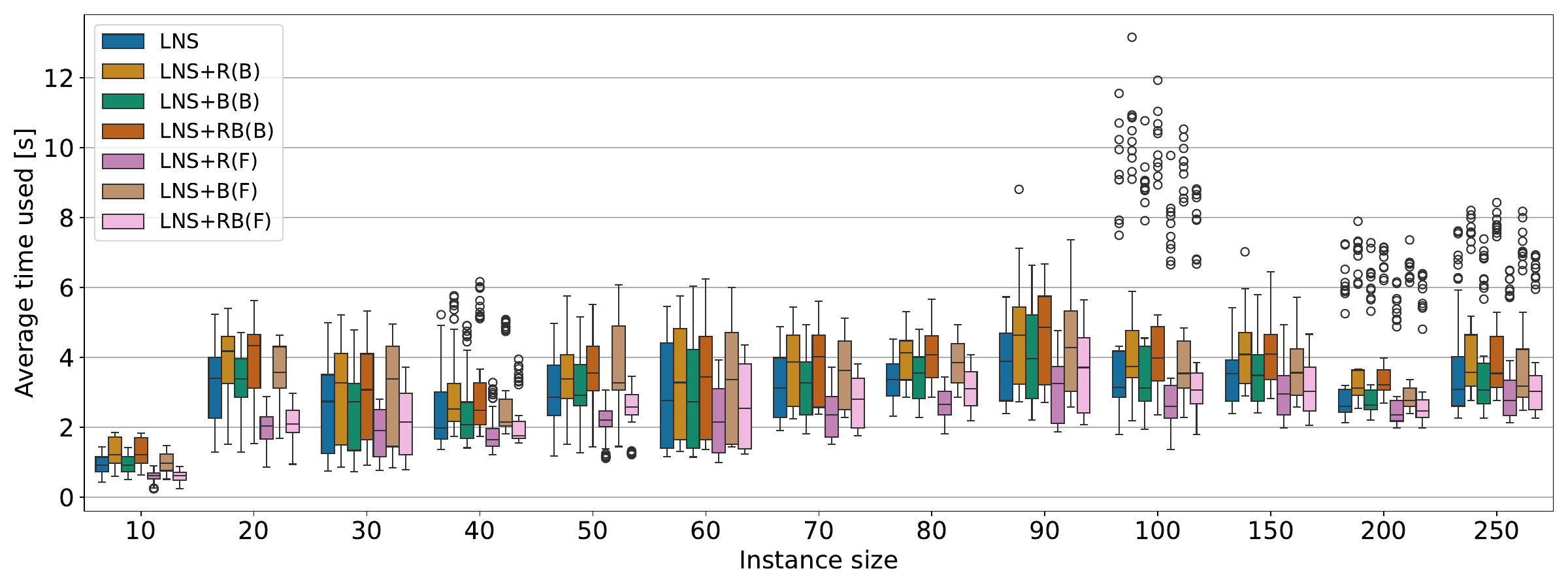}
        \caption{Average repairing time}\label{fig:boxplot_avgtime}
    \end{subcaptionblock}\\
    % SUCCESS RATE
    \begin{subcaptionblock}{\textwidth}
        \centering
        \includegraphics[width=\scalingFactor\textwidth]{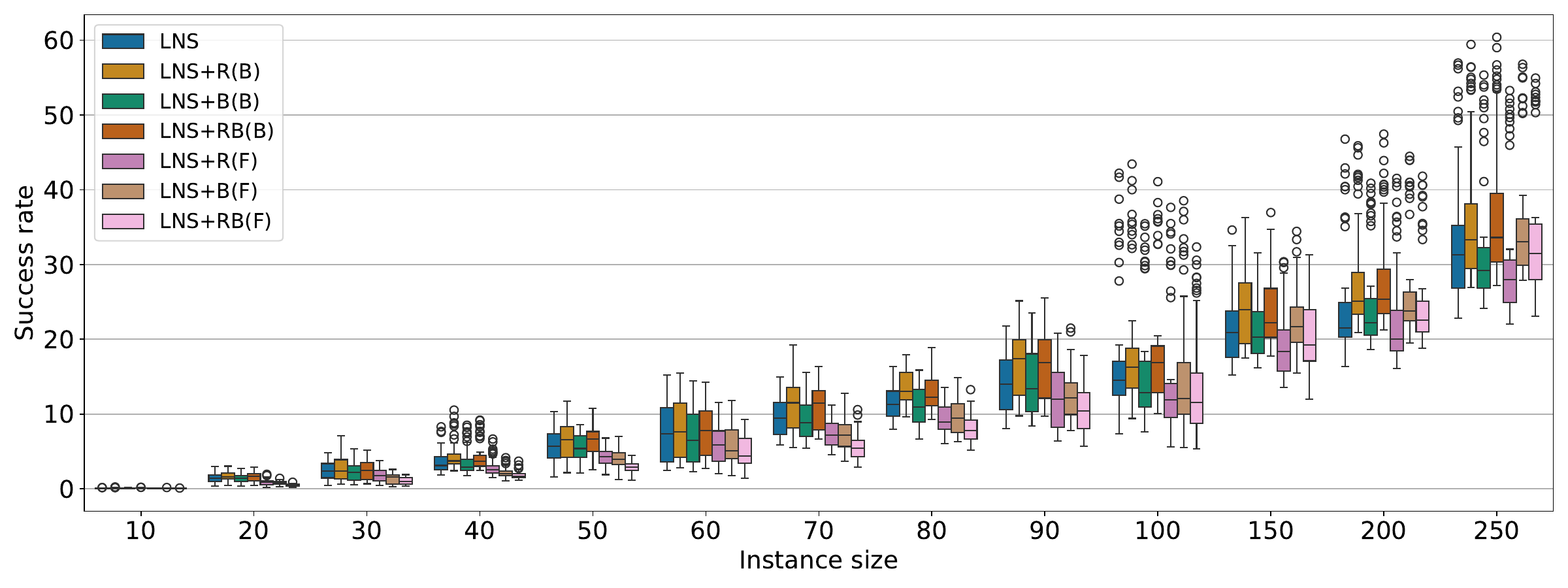}
            \caption{Success rate}\label{fig:boxplot_success_rate}
    \end{subcaptionblock}%
    \caption{Metrics of different LNS integrations}
    \label{fig:factors_size}
\end{figure}

Figure~\ref{fig:factors_size} shows the number of iterations, average repairing time, and success rate for each of the methods, grouped by size. The general trend based on the size is clear. Smaller instances allow more iterations, since they are faster. For smaller instances, the success rate is low, indicating fast convergence, while for larger instances success rate is high, indicating less time is spent in local optima.

Between different variants of \gls{lns}, differences are small, but with several notable patterns. First, \lns and both \textsc{+b} variants show very similar results in all three metrics, which makes sense since the work in the background thread should not have significant impact on the main thread.

Next, both \lnsr and \lnsrb show fewer and slower iterations, while both \lnsrf and \lnsrbf show more and faster iterations. This is interesting since column reuse (\textsc{+r}) is supposed to speed up solving the sub-problems. This only seems to work when adding the full set of columns \textsc{(f)}, while only adding the best columns \textsc{(b)} actually seems to generate overhead instead of speeding up the search.

\subsubsection{Performance}

\begin{figure}[!tph]
    \centering
    \includegraphics[width=\scalingFactor\textwidth]{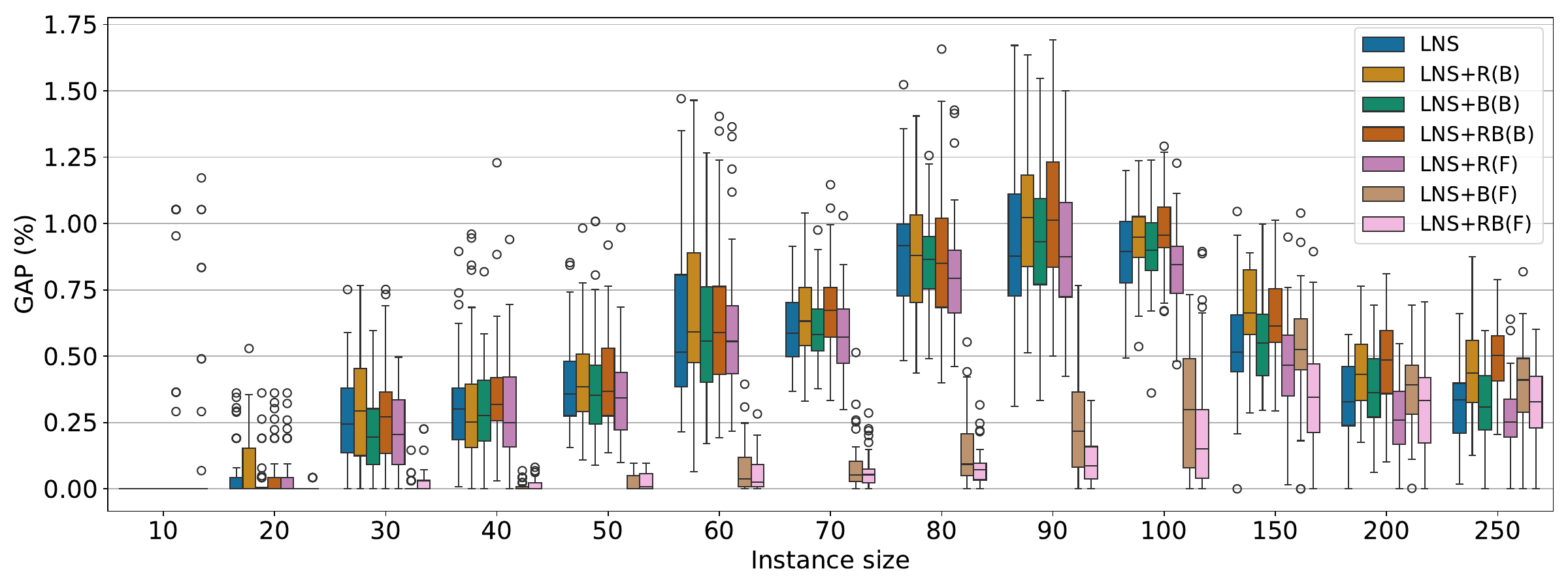}
            \caption{GAP for different LNS integrations}\label{fig:boxplot_gap}
\end{figure}

Figure~\ref{fig:boxplot_gap} shows the GAP to the best known solution for the different variants. While the distinction regarding the metrics was small, the distinction regarding the GAP is very clear. \lnsbf and \lnsrbf significantly outperform all other methods, which show similar performance among each other. This indicates that adding the background thread (\textsc{+b}) with the full set of columns \textsc{(f)} is the key combination to improve performance.

The shape of the graph shows that this improvement is most beneficial for medium to large, but not very large instances. For small instances, all methods perform very well, with more notable distinctions starting at size 30. Separation grows larger up to around size 90, while for larger instances the gap between methods shrinks, and methods perform very similar for sizes 200 and 250.

Of note is that for up to size 90, \lnsbf and \lnsrbf also show very low standard deviations, making them more stable than the other methods. \lnsbf starts to degrade around size 80, while \lnsrbf shows further benefits regarding deviations and also slightly better results than \lnsbf for sizes 80 to 150.

\begin{figure}[!tph]
    \centering
    \includegraphics[width=\scalingFactor\textwidth]{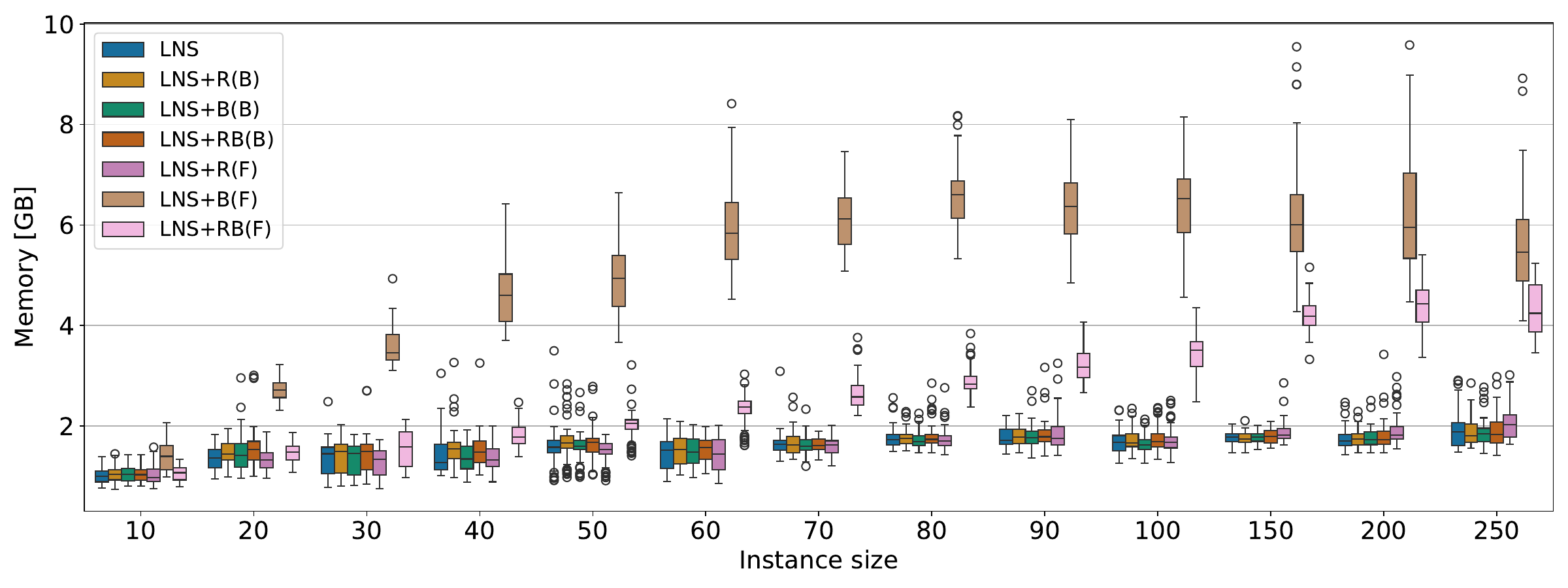}
            \caption{Memory usage of different LNS integrations}\label{fig:boxplot_ram_memory}
\end{figure}

Figure~\ref{fig:boxplot_ram_memory} shows the memory use of the different \gls{lns} versions. The first conclusion is that all versions that do not combine the background thread (\textsc{+b}) with full storage \textsc{(f)} show very similar memory use. The pure storage, even of the full set of columns, is not significant, since the storage using bit sets in the implementation is very light-weight. However, there is significantly larger memory usage for \lnsbf and \lnsrbf, especially for larger instances, showing that using the background thread with a larger number of columns is what needs extra memory.

While the quality of the results showed the large impact of the background thread and only minor improvements for column reuse, in comparison \lnsrbf uses significantly less memory than \lnsbf. Upon closer investigation, the number of columns in the background thread grows much slower in \lnsrbf than in \lnsbf. For example, in the instance realistic\_100\_49, the average number of columns differs by a factor of more than 3, from $\num{2.0d6}$ in \lnsbf to \num{6.7d5} in \lnsrbf. The suspected reason is that reusing the columns prevents the sub-problems from generating a large number of suboptimal columns which would clutter the background thread. This is also consistent with the slight performance benefit of \lnsbf regarding the GAP for sizes 80 to 150. While \lnsbf already has the background thread cluttered with too many sub-optimal columns, \lnsrbf avoids this bottleneck for longer.

\subsubsection{Convergence Plots}

\begin{figure}[!tph]
    \centering
        \begin{subcaptionblock}{.45\textwidth}
        \centering
        \includegraphics[width=\textwidth]{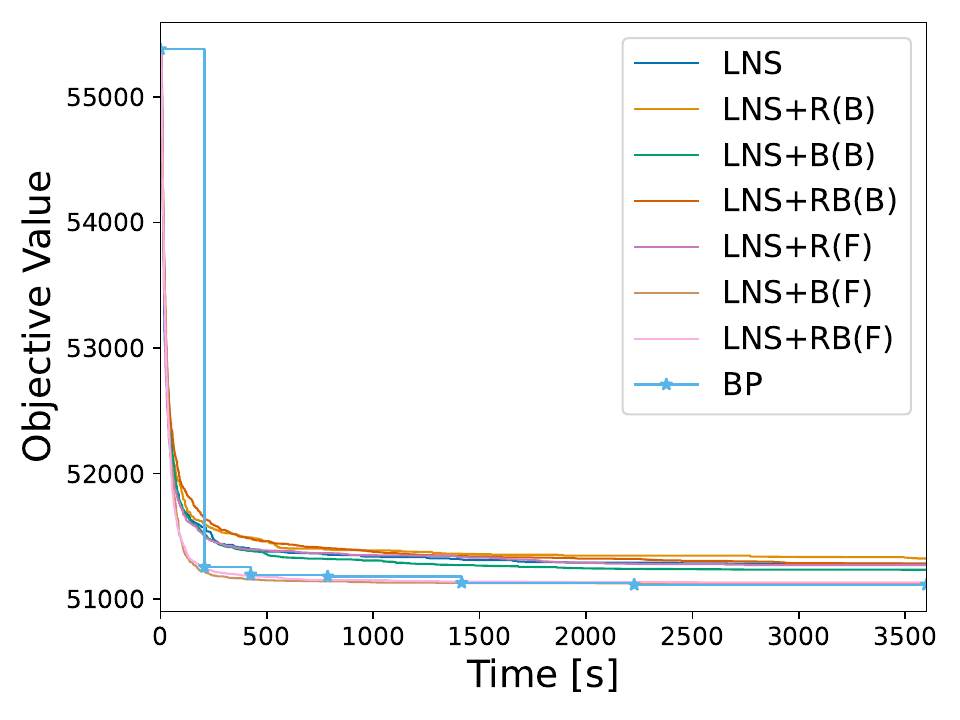}
            \caption{Realistic\_30\_12}\label{fig:convergence_30_12}
    \end{subcaptionblock}
    \begin{subcaptionblock}{.45\textwidth}
        \centering
        \includegraphics[width=\textwidth]{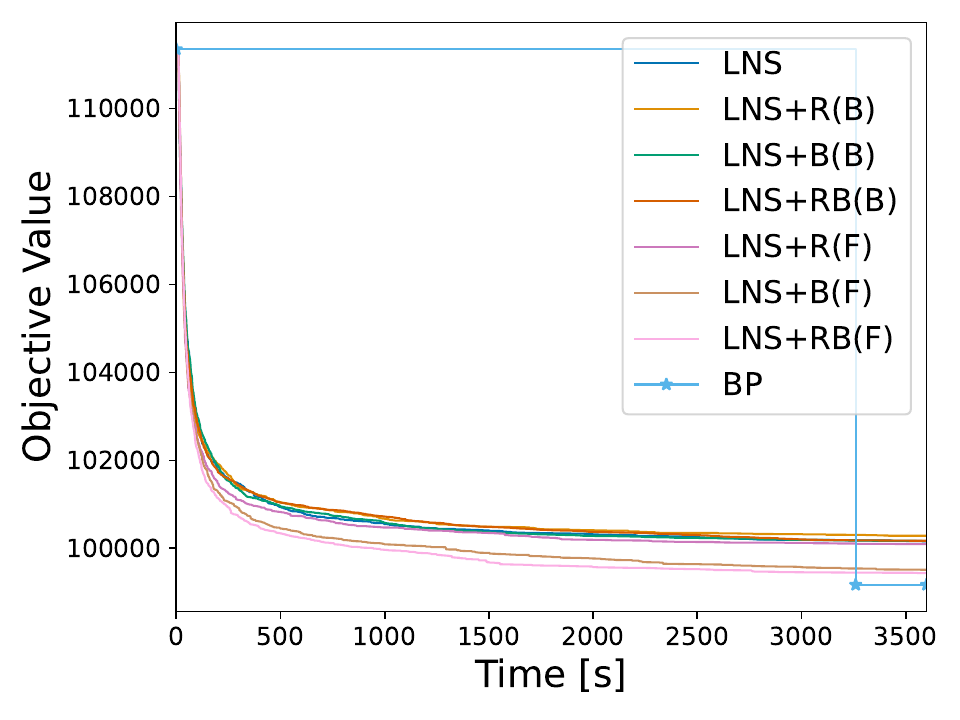}    \caption{Realistic\_60\_28}\label{fig:convergence_60_28}
    \end{subcaptionblock}\\
        \begin{subcaptionblock}{.45\textwidth}
        \centering
        \includegraphics[width=\textwidth]{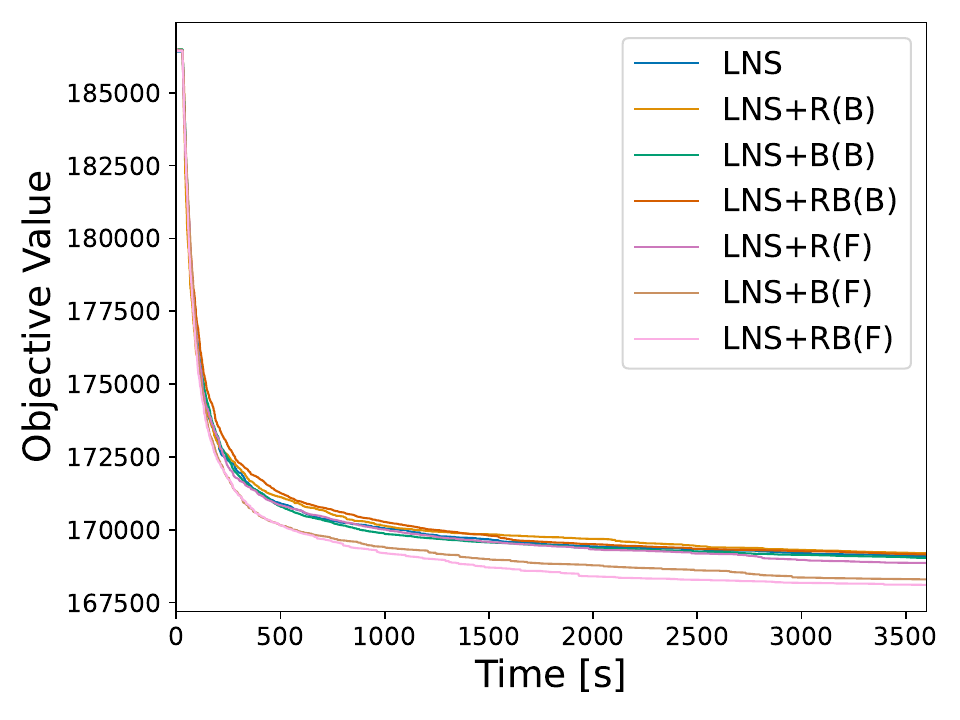}
            \caption{Realistic\_100\_49}\label{fig:convergence_100_49}
    \end{subcaptionblock}
    \begin{subcaptionblock}{.45\textwidth}
        \centering
        \includegraphics[width=\textwidth]{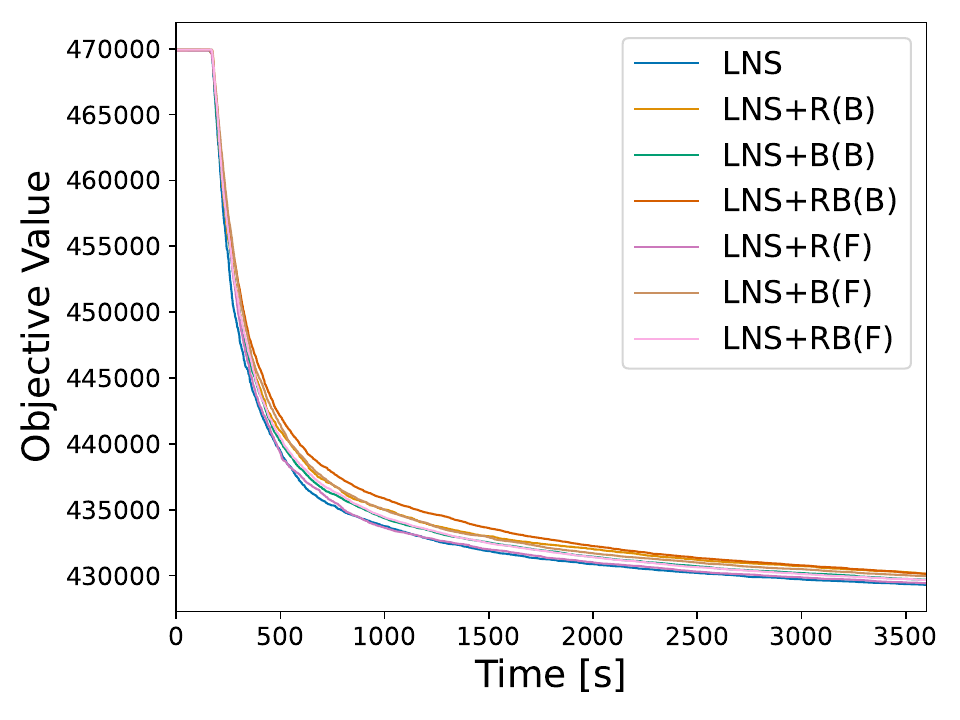}    \caption{Realistic\_250\_64}\label{fig:convergence_250_64}
    \end{subcaptionblock}\\
    \caption{Convergence plots}
    \label{fig:convergence plots}
\end{figure}

In order to analyze the behavior of the different methods over their runtime, Figure~\ref{fig:convergence plots} shows the convergence plots for some instances of different sizes to highlight different behaviors. The trajectories of the \gls{lns} versions are the average trajectories over 10 runs to avoid showing outlying behavior. 

The first three instances all show the dominance of \lnsbf and \lnsrbf over the other \gls{lns} methods. The trajectories of these methods separate from the others very early in the search, and consistently stay ahead until the end.

Instance 30\_12 shows the behavior of \bpplain for smaller instances. While the results increase in steps (at the end of some nodes in the branching tree), the overall result and trajectory are comparable to the best \gls{lns} versions. Instance 60\_28 shows the behavior of \bpplain for mid-sized instances. While it can outperform the \gls{lns} versions in the end, it does not provide a solution for a very long time, leading to a much better any-time behavior of \gls{lns}. For the larger instances, \bpplain only provides a solution with extra runtime, or using the background thread (\bpbg).

Instance 250\_64 shows the behavior of \gls{lns} for a very large instance. Here, the versions perform very similar, and the trajectories flatten out less, indicating that \gls{lns} is still improving regularly without hitting too many local optima, while the problem solved in the background thread gets larger and therefore harder to solve.

\subsection{Comparison of Different Methods}

\begin{table}[!th]
\centering
\caption{Results for different solution methods grouped by size}
\label{tab:results_size}
%\resizebox{\textwidth}{!}{%
\begin{tabular}{rS[table-format=6.1]*{3}{S[table-format=6.1]S[table-format=6.1]}}
\toprule
 & \bpplain & \multicolumn{2}{c}{\lns} & \multicolumn{2}{c}{\lnsrbf} & \multicolumn{2}{c}{\textsc{cmsa}}  \\
\cmidrule(lr){3-4}\cmidrule(lr){5-6}\cmidrule(lr){7-8}
Size & & {Min} & {Avg} & {Min} & {Avg} & {Min} & {Avg} \\
\midrule
10 & \bfseries 14709.2 & 14709.2 & \bfseries 14709.2 & 14709.2 & 14728.2 & 14867.4 & 14879.7 \\
20 & 30299.4 & 30294.6 & 30312.3 & 30294.6 & \bfseries 30295.9 & 30695.8 & 30745.9 \\
30 & \bfseries 49846.4 & 49867.0 & 49966.4 & 49841.0 & 49851.2 & 50731.4 & 50817.2 \\
40 & 67017.0 & 67023.8 & 67165.7 & 66957.6 & \bfseries 66966.7 & 68394.8 & 68499.9 \\
50 & 84338.8 & 84415.4 & 84583.7 & 84251.4 & \bfseries 84272.2 & 86219.0 & 86389.2 \\
60 & 99754.6 & 100006.4 & 100246.4 & 99673.6 & \bfseries 99704.7 & 102596.2 & 102822.9 \\
70 & 118337.6 & 118512.2 & 118698.9 & 117991.2 & \bfseries 118051.3 & 120935.6 & 121141.9 \\
80 & 134925.8 & 134827.6 & 135178.8 & 133983.4 & \bfseries 134081.6 & 138406.8 & 138760.3 \\
90 & 150292.8 & 150282.6 & 150721.6 & 149340.2 & \bfseries 149494.4 & 154692.6 & 155078.3 \\
100 & 168554.2 & 166397.4 & 166811.5 & 165331.4 & \bfseries 165698.9 & 171159.4 & 171786.7 \\
150 & 273629.2 & 254742.2 & 255449.1 & 254195.8 & \bfseries 255017.6 & 263079.2 & 263387.7 \\
200 & 380518.0 & 337479.0 & 338226.9 & 337221.8 & \bfseries 338161.2 & 348608.6 & 349017.0 \\
250 & 520063.4 & 426908.4 & 427866.7 & 426739.4 & \bfseries 427804.0 & 438811.4 & 439234.5 \\
\bottomrule
\end{tabular}
%}
\end{table}

\begin{figure}[!tph]
    \centering
    \includegraphics[width=\scalingFactor\textwidth]{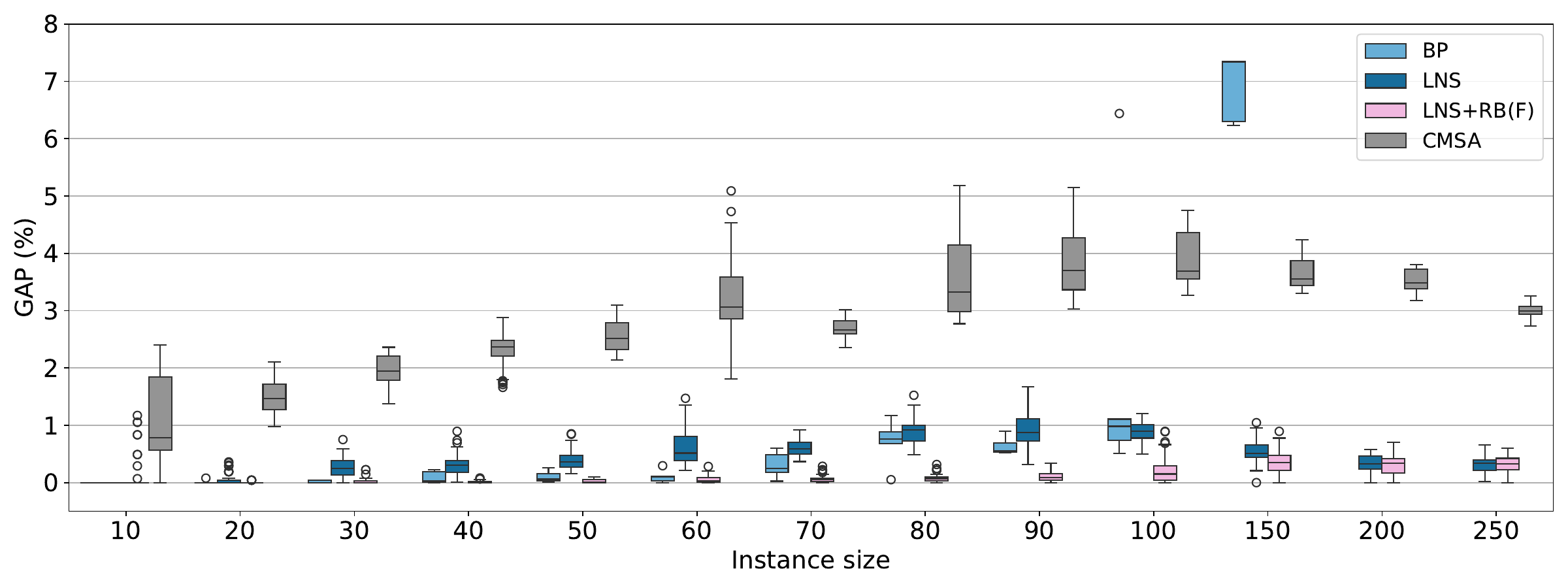}
    \caption{Grouped results }
    \label{fig:final_results_gap_with_bp_cmsa}
\end{figure}

Table~\ref{tab:results_size} shows the comparison of the final versions of \bpplain and \lns, and the best performing integration \lnsrbf, as well as the previous state of the art \textsc{cmsa} \cite{Rosati2023}. Note that \textsc{cmsa} outperforms all earlier meta-heuristics on this problem, and was executed on a more powerful machine than ours (AMD Ryzen Threadripper PRO 3975WX processor with 32 cores, with
a base clock frequency of 3.5 GHz, \qty{64}{\giga\byte} of RAM). For the stochastic methods, the minimum and average across 10 runs are presented, for \bpplain the result of one run. The table shows minimum and average per size, with each row averaging the 5 instances of this particular size. Detailed results for these and other tested methods per instance are available in the appendix. The best performing method is highlighted in bold. Figure~\ref{fig:final_results_gap_with_bp_cmsa} shows the gaps for these methods (average is used for stochastic methods).

The results indicate that \bpplain is the best method to use for small instances, as it can solve very small instances to optimality in a few seconds, and provides high-quality solutions with known low gaps to the optimum for up to size 30. \bpplain considerably starts to struggle for large instances.

While results of \bpplain still outperforms \textsc{cmsa} for instances of size up to 90, starting with size 40, the integrated method \lnsrbf already outperforms all other methods, and consistently provides the best results across all further sizes, making it the new state of the art for this problem.

Since \lnsrbf uses the background thread, and therefore more resources both regarding computation and memory, the best method regarding efficiency is \lns, which only needs one thread, and uses very moderate memory consumption. Still, in a very competitive environment, small improvements like those gained by \lnsrbf translate to large savings over time, and the moderate additional resource consumption is worth the additional benefit in practice.

\subsubsection{Statistical Significance.}

To add a well-founded view of the results, we perform a statistical analysis. For this purpose, we use the software the tool \textsc{scmamp}~\cite{scmamp}\footnote{\url{https://github.com/b0rxa/scmamp/tree/3cf4d8b9759769cdf20771afa0efc33a5265c7f9}} ran with \texttt{R} (version 4.2.2). For our comparison, we consider \bpplain, all our variants of LNS, as well as other results from the literature: \textsc{cmsa}~\cite{Rosati2023}, Simulated Annealing,  Hill Climbing~\cite{kletzander2020solving}, and Tabu Search~\cite{Kletzander2022}.

The analysis is done in two steps: the omnibus test, and post-hoc test. In the omnibus test, we check whether at least one of the algorithms performs differently than the others. To do that, following the guidelines of Calvo \cite{Calvo}, we use the Friedman test with Iman and Davemport extension.  We formulate a Hypothesis $H_0$: \textit{for every instance, the average objective function values are identical on all the algorithms.}  With a significance level of $\alpha=0.05$, the test rejects the null hypothesis with a $p$-value smaller than $\num{2.2d-16}$.
Therefore, we can conclude that that there is strong statistical evidence that at least one algorithm performs differently than the rest. 
Thus, we conduct a post-hoc test to detect differences by pairs.

% CRITICAL DIFFERENCE PLOT
We conduct an all pair-wise comparison with the Shaffer's static method. The results on all instances are graphically shown with a \textit{Critical Difference} (CD) plot in Figure~\ref{fig:CDplots}. Each considered algorithm is placed on the horizontal axis according to its average ranking for the instances (lower is better). The performances of those algorithm variants below the critical difference threshold ($0.94$) 
are considered statistically equivalent. In the CD plot, this is remarked by a horizontal bold bar that joins different algorithms.

\begin{figure}[!tph]
    \centering
\includegraphics[width=\textwidth]{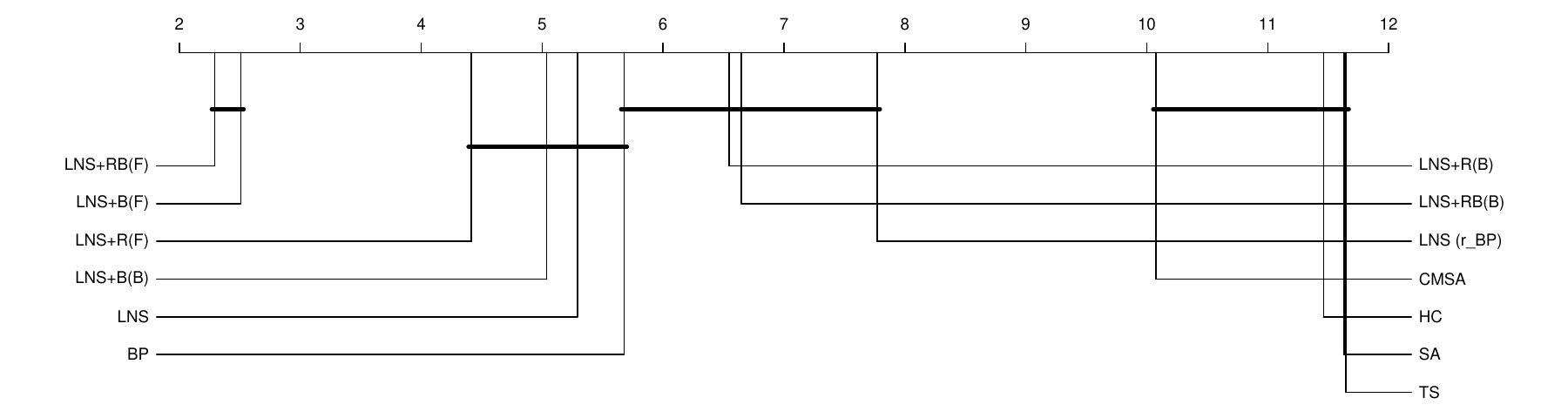}
    \caption{Critical Difference plot}
    \label{fig:CDplots}
\end{figure}

According to Figure~\ref{fig:CDplots} there is no significant difference between \lnsrbf and \lnsbf. They both outperform all other algorithms and LNS variants. All variants of \gls{lns} outperform the previous meta-heuristic methods with significant difference.

\section{Conclusions}\label{sec:conclusion}

In this article we provided a comprehensive study of \acrfull{bp} and \acrfull{lns} for the \acrfull{bdsp}. We defined a high-dimensional \acrfull{rcspp} as the sub-problem for \gls{bp} that is able to represent the complex break constraints in our model, and introduced several novel adaptations to solve this sub-problem efficiently, including splitting the sub-problem, exponential arc throttling, and the usage of k-d trees. With this method we are able to solve small problems to optimality in seconds, and provide results with known optimality gaps of less than one percent for mid-sized instances, constituting the state of the art for these instances.

For \gls{lns}, we introduced the novel destroy operator \tr that exploits the structure of the problem, and thoroughly evaluated several different design choices for both the destroy and repair operators, as well as the use of adaptivity. The results scale very well even to very large instances, and can outperform previous meta-heuristic solution methods.

Finally, we proposed a novel tight integration between \gls{lns} and \acrfull{cg}, where the best version of the integration \lnsrbf stores all columns from each sub-problem in \gls{lns}, reusing them for future sub-problems, and solving the integer master problem with all columns in a background thread. The evaluation showed that the background thread helps to significantly improve the results of \gls{lns}, especially for mid- to large-sized instances, while the column reuse is very beneficial to reduce the memory usage of the background thread. Overall, \lnsrbf is the new state of the art for instances of medium to large size.

While this article evaluates the methods on this complex version of \gls{bdsp}, both the concepts introduced to solve high-dimensional \glspl{rcspp}, and the integration of \gls{lns} and \gls{cg} are general and can be applied to other scenarios, or to other optimization problems.

As future work, we would like to apply the new concepts to different optimization problems, both Bus Driver Scheduling with different rule sets, as well as entirely different problems. For the \gls{bdsp}, future work could focus especially of the very large instances, where the benefits of \lnsrbf seem to degrade, and investigate further improvements of the integration like eliminating sub-optimal columns from the background thread, or more choices between accepting the best or full set of columns.

% \begin{acks}
% This project is partially funded by the Doctoral Program Vienna Graduate School on Computational Optimization (VGSCO), Austrian Science Foundation (FWF), under grant No W1260-N35. The financial support by the Austrian Federal Ministry of Labour and Economy, the National Foundation for Research, Technology and Development and the Christian Doppler Research
% Association is gratefully acknowledged. The project was also supported
% by a NSF Leap-HI award NSF-1854684.
% \end{acks}

{\small
\bibliographystyle{alphaurl}
\bibliography{bibliography}
}
\appendix

\section{Reproducibility Checklist for JAIR}

All data artifacts and the part of the code we can provide are available online\footnote{\url{https://doi.org/10.5281/zenodo.15276063}}.

\subsection*{All articles:}

%\hh{revised for stylistic consistency:}
\begin{enumerate}
    \item All claims investigated in this work are clearly stated. 
    [yes]
    \item Clear explanations are given how the work reported substantiates the claims. 
    [yes]
    \item Limitations or technical assumptions are stated clearly and explicitly. 
    [yes]
    \item Conceptual outlines and/or pseudo-code descriptions of the AI methods introduced in this work are provided, and important implementation details are discussed. 
    [yes]
    \item 
    Motivation is provided for all design choices, including algorithms, implementation choices, parameters, data sets and experimental protocols beyond metrics.
    [yes]
\end{enumerate}

\subsection*{Articles containing theoretical contributions:}
Does this paper make theoretical contributions? 
[no] 

\subsection*{Articles reporting on computational experiments:}
Does this paper include computational experiments? [yes]

If yes, please complete the list below.
\begin{enumerate}
    \item 
    All source code required for conducting experiments is included in an online appendix 
    or will be made publicly available upon publication of the paper.
    The online appendix follows best practices for source code readability and documentation as well as for long-term accessibility.
    [partially]
    \item The source code comes with a license that
    allows free usage for reproducibility purposes.
    [partially]
    \item The source code comes with a license that
    allows free usage for research purposes in general.
    [partially]
    \item 
    Raw, unaggregated data from all experiments is included in an online appendix 
    or will be made publicly available upon publication of the paper.
    The online appendix follows best practices for long-term accessibility.
    [yes]
    \item The unaggregated data comes with a license that
    allows free usage for reproducibility purposes.
    [yes]
    \item The unaggregated data comes with a license that
    allows free usage for research purposes in general.
    [yes]
    \item If an algorithm depends on randomness, then the method used for generating random numbers and for setting seeds is described in a way sufficient to allow replication of results. 
    [yes]
    \item The execution environment for experiments, the computing infrastructure (hardware and software) used for running them, is described, including GPU/CPU makes and models; amount of memory (cache and RAM); make and version of operating system; names and versions of relevant software libraries and frameworks. 
    [yes]
    \item 
    The evaluation metrics used in experiments are clearly explained and their choice is explicitly motivated. 
    [yes]
    \item 
    The number of algorithm runs used to compute each result is reported. 
    [yes]
    \item 
    Reported results have not been ``cherry-picked'' by silently ignoring unsuccessful or unsatisfactory experiments. 
    [yes]
    \item 
    Analysis of results goes beyond single-dimensional summaries of performance (e.g., average, median) to include measures of variation, confidence, or other distributional information. 
    [yes]
    \item 
    All (hyper-) parameter settings for 
    the algorithms/methods used in experiments have been reported, along with the rationale or method for determining them. 
    [yes]
    \item 
    The number and range of (hyper-) parameter settings explored prior to conducting final experiments have been indicated, along with the effort spent on (hyper-) parameter optimisation. 
    [yes]
    \item 
    Appropriately chosen statistical hypothesis tests are used to establish statistical significance
    in the presence of noise effects.
    [yes]
\end{enumerate}

\subsection*{Articles using data sets:}
Does this work rely on one or more data sets (possibly obtained from a benchmark generator or similar software artifact)? 
[yes]

If yes, please complete the list below.
\begin{enumerate}
    \item 
    All newly introduced data sets 
    are included in an online appendix 
    or will be made publicly available upon publication of the paper.
    The online appendix follows best practices for long-term accessibility with a license
    that allows free usage for research purposes.
    [NA]
    \item The newly introduced data set comes with a license that
    allows free usage for reproducibility purposes.
    [NA]
    \item The newly introduced data set comes with a license that
    allows free usage for research purposes in general.
    [NA]
    \item All data sets drawn from the literature or other public sources (potentially including authors' own previously published work) are accompanied by appropriate citations.
    [yes]
    \item All data sets drawn from the existing literature (potentially including authors’ own previously published work) are publicly available. [yes]
    \item All data sets that are not publicly available are described in detail.
    [NA]
    \item All new data sets and data sets that are not publicly available are described in detail, including relevant statistics, the data collection process and annotation process if relevant.
    [NA]
    \item 
    All methods used for preprocessing, augmenting, batching or splitting data sets (e.g., in the context of hold-out or cross-validation)
    are described in detail. [yes]
\end{enumerate}

\subsection*{Explanations on any of the answers above (optional):}

Code availability: The Branch\&Price source code was developed in collaboration with our company partner and cannot be shared.

\end{document}